\documentclass[11pt,a4paper]{article}

\title{\bf Globally linked pairs and cheapest globally rigid supergraphs}
\author{
Tibor Jord\'an\thanks{Department of Operations Research, ELTE E\"otv\"os Lor\'and University, and HUN-REN--ELTE Egerv\'ary Research Group
on Combinatorial Optimization, P\'azm\'any P\'eter s\'et\'any 1/C, 1117 Budapest, Hungary.
e-mail: {\tt tibor.jordan@ttk.elte.hu}} 
\and
Soma Vill\'anyi\thanks{HUN-REN--ELTE Egerv\'ary Research Group
on Combinatorial Optimization,
P\'azm\'any P\'eter s\'et\'any 1/C, 1117 Budapest, Hungary.
e-mail: {\tt villanyi.soma@gmail.com}}
}
\date{January 14, 2024}

\usepackage[utf8]{inputenc}
\usepackage[english]{babel}
\usepackage{amsmath}
\usepackage{amsthm}
\usepackage{amssymb}
\usepackage{geometry}
\usepackage{a4wide}
\usepackage{bm}
\usepackage{microtype}
\usepackage{mathtools}
\usepackage{csquotes}
\usepackage[shortlabels]{enumitem}
\usepackage{algorithm}
\usepackage{algpseudocode}
\usepackage{caption}
\usepackage{thmtools}
\usepackage{thm-restate}

\theoremstyle{definition}
\newtheorem{theorem}{Theorem}[section]
\newtheorem{lemma}[theorem]{Lemma}
\newtheorem*{lemma*}{Lemma}
\newtheorem*{conjecture*}{Conjecture}
\newtheorem*{lemma''*}{``Lemma''}

\newtheorem*{claim*}{Claim}
\newtheorem{corollary}[theorem]{Corollary}

\newcommand{\R}{\mathbb{R}}

\newcommand{\tlc}{{\rm tlc}}
\newcommand{\cll}{{\rm sp}}

\providecommand{\given}{}
\newcommand{\setseparator}{%
  \mathrel{}
  \mathclose{}
  \delimsize|
  \mathopen{}
  \mathrel{}
}
\DeclarePairedDelimiterX{\Set}[1]{( }{ ) }{%
  \renewcommand{\given}{\setseparator}%
  #1%
}
\newcommand{\longsetdescription}[2][.55\displaywidth]{\parbox{#1}{#2}}
\newcommand{\longsetdescriptiontwo}[2][.45\displaywidth]{\parbox{#1}{#2}}

\begin{document}
\maketitle

\begin{abstract}
Given a graph $G$, a cost function on the non-edges of $G$, and an integer $d$, the problem of finding a cheapest globally rigid supergraph of $G$ in $\mathbb{R}^d$ is NP-hard for $d\geq 1$. For this problem, which is a common generalization of several well-studied graph augmentation problems, no approximation algorithm has previously been known for $d\geq 2$. Our main algorithmic result is a 5-approximation algorithm in the $d=2$ case. We achieve this by proving numerous new structural results on rigid graphs and globally linked vertex pairs. In particular, we show that every rigid graph in $\mathbb{R}^2$ has a tree-like structure, which conveys all the information regarding its globally rigid augmentations. Our results also yield a new, simple solution to the minimum cardinality version (where the cost function is uniform) for rigid input graphs, a problem which is known to be solvable in polynomial time.
\end{abstract}

\section{Introduction}

A $d$-dimensional {\it framework}
is a pair $(G,p)$, where $G=(V,E)$
is a graph and $p$ is a map from $V$ to $\mathbb{R}^d$.
We also say that $(G,p)$ is a {\it realization} of $G$ in
$\mathbb{R}^d$. The {\it length} of an edge $uv\in E$ in
$(G,p)$ is $||p(u)-p(v)||$, where $||.||$ denotes the Euclidean norm in $\mathbb{R}^d$.
Two frameworks
$(G,p)$ and $(G,q)$ are {\it equivalent} if corresponding edge lengths are the same, that is, 
$||p(u)-p(v)||=||q(u)-q(v)||$ holds
for all pairs $u,v$ with $uv\in E$.
The frameworks $(G,p)$ and $(G,q)$ are {\it congruent} if
$||p(u)-p(v)||=||q(u)-q(v)||$ holds
for all pairs $u,v$
with $u,v\in V$.

A $d$-dimensional framework $(G,p)$ is called {\it globally rigid} if every equivalent 
$d$-dimensional framework
$(G,q)$ is congruent to $(G,p)$. This is the same as saying that the edge lengths of $(G,p)$
uniquely determine all the pairwise distances.
It is NP-hard to test whether a given framework in $\R^d$ is globally
rigid, even for $d=1$ \cite{Saxe}.
This fundamental property of frameworks becomes more tractable if we consider generic
frameworks. A framework $(G,p)$ %
 is said to be {\it generic} if the set of
its $d|V(G)|$ vertex coordinates is algebraically independent over $\mathbb{Q}$.
It is known that in a given dimension 
the global rigidity of a generic framework $(G,p)$ depends only on $G$: either every generic realization of $G$ in $\R^d$ is globally rigid, or none of them are
\cite{Con, GHT}.
Thus, we say that a graph $G$ is {\it globally rigid} in $\R^d$ if every (or equivalently, if some) $d$-dimensional
generic realization of $G$ is globally rigid in $\R^d$.
For $d=1,2$, combinatorial characterizations and corresponding
deterministic polynomial time algorithms are known for (testing) 
global rigidity in $\R^d$.
The case $d=1$   is a folklore result:
it is not hard to see that a graph $G$ on at least three vertices is globally rigid in $\R^1$ if and only
if it is $2$-connected. For the case $d=2$, see Theorem \ref{thm}.
The existence of such a characterization
(or algorithm) for $d\geq 3$ is a major open question. 
Global rigidity is relevant in several modern applications (e.g., sensor network localization, molecular conformation), which  provides additional motivation for the study of the related optimization problems.

In this paper we consider the ($d$-dimensional)  {\it cheapest globally rigid supergraph} problem. In this problem
the input is a graph $G=(V,E)$, a cost function $c:V\times V\to \R^+\cup \{\infty\}$, and a positive integer $d$.
The goal is to 
find an edge set $E'$ on vertex set $V$ of minimum total cost, for which $G'=(V,E\cup E')$ is globally rigid in $\R^d$.
In other words, we look for a cheapest augmentation which makes $G$ globally rigid.
This problem is known to be NP-hard for all $d\geq 1$, even for metric cost functions, see \cite{JorMih}.
For $d=1$ the problem is equivalent to the well-studied cheapest 2-connected supergraph problem, for which
Khuller and Raghavachari \cite{KhRa} gave a 2-approximation algorithm.
In the case when $d=2$, the input graph has no edges and the cost function is metric (resp. if each cost is either 1 or $\infty$),
a 2-approximation (resp. $\frac{3}{2}$-approximation) algorithm was given in \cite{JorMih}. In the same
paper %
a constant factor approximation algorithm was given for every fixed $d\geq 3$,
for metric cost functions.
However, no approximation algorithms have been found for the general version for $d\geq 2$.

The main algorithmic result of this paper is 
a 5-approximation algorithm for the 2-dimensional cheapest globally rigid supergraph problem.
We also consider the case of uniform edge costs, which we call the {\it minimum size globally rigid
supergraph} problem.
Kir\'aly and Mih\'alyk\'o \cite{KMsidma} proved that this special case
can be solved in polynomial time for rigid input graphs in $\R^2$.
We shall revisit this problem and give a new algorithm and a shorter proof for a simplified
min-max theorem, based on the methods developed for the minimum cost version.

In order to obtain these algorithmic results, we prove numerous new structural results on
rigid graphs and so-called weakly globally linked vertex pairs, which may be of independent interest. 
In particular, 
we show the surprising fact that every rigid graph in $\R^2$ has a tree-like structure, which conveys all the information regarding its globally rigid augmentations.
This result is one of the main theoretical contributions of this paper. 
The proof relies on and extends the recent results on weakly globally linked pairs from \cite{wgl}. 
We can represent this tree-like structure by a tree, and use this tree 
to reduce our augmentation problems to specific augmentation problems on trees.
We also introduce and study several graph families, which turn out to be important in the analysis of rigid and globally rigid graphs, as well as
weakly globally linked pairs. In some cases we state 
the higher dimensional generalizations of our results, which may find applications in related problems 
of rigidity theory.

It turns out that the tree representation of rigid graphs can also be used in connectivity augmentation problems with chordal
input graphs.
As we shall show, a supergraph of a $k$-connected chordal graph is $(k+1)$-connected if and only if it is globally rigid in $\R^{k}$.
This will enable us to design a 2-approximation algorithm for the cheapest $(k+1)$-connected supergraph problem for $k$-connected chordal graphs.
Interestingly, the cardinality case is again polynomial time solvable: this was first pointed out in \cite{JJ} in a more general context.
A direct approach was developed in \cite{NaSa}.

The rest of the paper is organized as follows. In Section \ref{sec:pre} we introduce the necessary notions and results concerning rigid graphs and frameworks.
In Section \ref{sec:wgl} we introduce weakly globally linked pairs and show how we can use them to
reduce the globally rigid supergraph problem to so-called totally loose input graphs. 
In Section \ref{sec:Tloose}
we prove some key lemmas %
about totally loose graphs and describe their tree structure. %
In Section \ref{sec:augmenations} we describe the globally rigid supergraphs of rigid graphs.
In \mbox{Section \ref{sec:cheapest}} we give the 5-approximation algorithm for
the minimum cost globally rigid supergraph problem. In Section \ref{sec:minsize}
we study the minimum size version.
The results on chordal graphs and the algorithmic aspects are given in Sections \ref{sec:chordal} and \ref{sec:algo}, respectively.

\section{Preliminaries}
\label{sec:pre}

In the structural results on global rigidity and global linkedness the notions of
rigid frameworks, rigid graphs and the rigidity matroid play a key role.

A $d$-dimensional framework $(G,p)$ is {\it rigid} if there exists some $\varepsilon
>0$ such that, if $(G,q)$ is equivalent to $(G,p)$ and
$||p(v)-q(v)||< \varepsilon$ for all $v\in V$, then $(G,q)$ is
congruent to $(G,p)$. 
This is equivalent to requiring that every continuous motion of the vertices of $(G,p)$ in $\R^d$ that
preserves the edge lengths takes the framework to a congruent realization of $G$.
It is known that in a given dimension 
the rigidity of a generic framework $(G,p)$ depends only on $G$: either every generic realization of $G$ in $\R^d$ is rigid, or none of them are \cite{AR}.
Thus, we say that a graph $G$ is {\it rigid} in $\R^d$ if every (or equivalently, if some) $d$-dimensional
generic realization of $G$ is rigid in $\R^d$.
For $d=1,2$, combinatorial characterizations and corresponding
deterministic polynomial time algorithms are known for (testing) rigidity in $\R^d$, see e.g. \cite{laman}.
The existence of such a characterization
(or algorithm) for $d\geq 3$ is a major open question.

Let $(G,p)$ be a realization of a graph $G=(V,E)$ in $\R^d$.
The \emph{rigidity matrix} of the framework $(G,p)$
is the matrix $R(G,p)$ of size
$|E|\times d|V|$, where, for each edge $uv\in E$, in the row
corresponding to $uv$,
the entries in the $d$ columns corresponding to vertices $u$ and $v$ contain
the $d$ coordinates of
$(p(u)-p(v))$ and $(p(v)-p(u))$, respectively,
and the remaining entries
are zeros. 
The rigidity matrix of $(G,p)$ defines
the \emph{rigidity matroid}  of $(G,p)$ on the ground set $E$
by linear independence of the rows. %
It is known that any pair of generic frameworks
$(G,p)$ and $(G,q)$ have the same rigidity matroid.
We call this the $d$-dimensional \emph{rigidity matroid}
${\cal R}_d(G)=(E,r_d)$ of the graph $G$.

We denote the rank of ${\cal R}_d(G)$ by $r_d(G)$.
A graph $G=(V,E)$ is \emph{${\cal R}_d$-independent} if $r_d(G)=|E|$ and it is an \emph{${\cal R}_d$-circuit} if it is not ${\cal R}_d$-independent but every proper 
subgraph $G'$ of $G$ is ${\cal R}_d$-independent. 
An edge $e$ of $G$ is an \emph{${\cal R}_d$-bridge in $G$}
if  $r_d(G-e)=r_d(G)-1$ holds. Equivalently, $e$ is an ${\cal R}_d$-bridge in $G$ if it is not contained in any subgraph of $G$ that is an ${\cal R}_d$-circuit.

The following characterization of rigid graphs is due to Gluck.
\begin{theorem}\label{theorem:gluck}
\cite{Gluck}
\label{combrigid}
Let $G=(V,E)$ be a graph with $|V|\geq d+1$. Then $G$ is rigid in $\R^d$
if and only if $r_d(G)=d|V|-\binom{d+1}{2}$.
\end{theorem}

A graph is \emph{minimally rigid} in $\R^d$ if it is rigid in $\R^d$ but 
$G-e$ is not rigid in $\R^d$ for
every edge $e$ of $G$. 
It follows from Theorem \ref{theorem:gluck} that every minimally rigid graph in $\R^d$ on $|V|\geq d+1$ vertices has exactly $d|V| - \binom{d+1}{2}$ edges.

Let $G=(V,E)$ be a graph and $\{u,v\}$ be a pair of vertices of $G$. 
An induced subgraph $G[X]$ (and the set $X$), for some $X\subseteq V$, is said to be
{\it $(u,v)$-rigid} in $\R^d$ (or  $\R^d$-$(u,v)$-rigid, or simply $(u,v)$-rigid, if $d$ is clear from the context),
if $G[X]$ is rigid in $\R^d$ and $u,v\in X$. Similarly, for a vertex set $V_0\subseteq V$,  $G[X]$ (and the set $X$) is said to be $V_0$-rigid in $\R^d$ if $G[X]$ is rigid in $\R^d$ and $V_0\subseteq X$.
The pair $\{u,v\}$ is called {\it linked} in $G$ in $\R^d$ if
$r_d(G+uv)=r_d(G)$ holds.  It is known that a pair
$\{u,v\}$ is linked in $G$ in $\R^2$ if and only if there exists a $(u,v)$-rigid subgraph of $G$.  For more details on the $2$-dimensional rigidity matroid, see \cite{Jmemoirs}.
A rigid graph $G$ (on at least two vertices) is called
{\it redundantly rigid} in $\R^d$ if $G-e$ is rigid
in $\R^d$
for
all $e\in E(G)$.

Let ${\cal M}$ be a matroid on ground set $E$. 
We can define a relation on the pairs of elements of $E$ by
saying that $e,f\in E$ are
equivalent if $e=f$ or there is a circuit $C$ of ${\cal M}$
with $\{e,f\}\subseteq C$.
This defines an equivalence relation. The equivalence classes are 
the \emph{connected components} of ${\cal M}$.
The matroid is \emph{connected} if 
it has only one connected component.
A graph $G=(V,E)$ is \emph{${\cal R}_d$-connected} if ${\cal R}_d(G)$ is connected.
We shall use the well-known fact that 
if $v$ is a vertex of degree at most $d$ in $G$, then every edge incident with $v$ is an
${\cal R}_d$-bridge in $G$. Hence the addition of a new vertex of degree $d$ to a rigid graph $G$ 
in $\R^d$ preserves rigidity (c.f. Theorem \ref{theorem:gluck}).

The characterization of globally rigid graphs in $\R^2$ is as follows.

\begin{theorem} \cite{JJconnrig}
\label{thm} Let $G$ be a graph on at least four vertices.
Then
$G$ is globally rigid in $\R^2$ if and only if
$G$ is $3$-connected and ${\cal R}_2$-connected.
\end{theorem}

For more details on globally rigid graphs and frameworks see \cite{JW}.

\section{Weakly globally linked pairs}
\label{sec:wgl}

Our approach is based on the notion of {\it (weakly) globally linked} pairs of vertices and the results of \cite{wgl}. 
A pair of vertices $\{u,v\}$ in a framework $(G,p)$ is called {\it globally linked in $(G,p)$} if for every equivalent framework $(G,q)$ we have
$||p(u)-p(v)||=||q(u)-q(v)||$. It was pointed out in \cite{JJS} that
global linkedness in $\R^d$ is not a generic property (for $d\geq 2$): 
a vertex pair may be globally linked in some generic $d$-dimensional realization of $G$ without being globally linked in all generic realizations. 
See Figure \ref{fig:gen1}. 
We say that
a pair $\{u,v\}$ is {\it globally linked in $G$} in $\R^d$ if it is globally linked in all generic $d$-dimensional frameworks $(G,p)$. 
We call a pair $\{u,v\}$ {\it weakly globally linked in $G$} in $\R^d$ if there exists a generic $d$-dimensional framework $(G,p)$
in which $\{u,v\}$ is globally linked. If  $\{u,v\}$ is not weakly globally linked in $G$, then it is called {\it globally loose} in $G$.
It is immediate from the definitions that $G$ is globally rigid in $\R^d$ if and only if all pairs of vertices of $G$ are globally linked in $G$ in $\R^d$.
We say that $G$ is {\it totally loose}, if 
$\{u,v\}$ is globally loose in $G$ for all $u,v\in V$ with $uv\notin E$.
This is equivalent to saying that in every generic realization $(G,p)$ the only
globally linked pairs are the adjacent vertex pairs.

\begin{figure}[t]
\begin{center}
\includegraphics[scale=1.1]{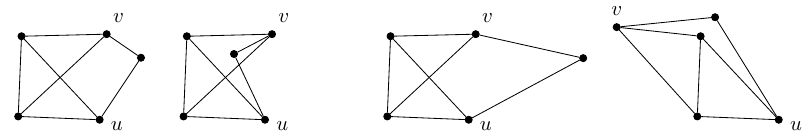}
\vspace{-0.5cm}
\end{center}
\caption{Two pairs of equivalent generic frameworks of a graph $G$ in $\R^2$. The vertex pair $\{u,v\}$ is globally linked in the two frameworks on the left. On the other hand, $\{u,v\}$ is not globally linked in the two frameworks on the right. Thus $\{u,v\}$ is weakly globally linked, but not globally linked in $G$ in $\R^2$.}
\label{fig:gen1}
\end{figure}

The case $d=1$ is exceptional and well-understood. Global linkedness in $\R^1$ is a generic property: a pair is globally linked
in $G$ in $\R^1$ if and only if there is a cycle in $G$ that contains both
vertices. Otherwise, it is globally loose. A graph is totally loose in $\R^1$ if and only if each of its 2-connected components is complete.
For $d\geq 3$ no combinatorial (or efficiently testable) characterization is known for globally linked or weakly globally linked pairs in graphs. 
In a recent paper \cite{wgl} we characterized the weakly globally linked pairs in graphs in the $d=2$ case and showed that
it can be tested in $O(|V|^2)$ time whether a given pair $\{u,v\}$ of vertices is weakly globally linked in a graph $G=(V,E)$ in $\R^2$.
This and some further results  are given in the following subsections.

\subsection{The totally loose closure of a graph}
We will need several new notions and structural results related to weakly globally linked and globally loose pairs.
We define 
the {\it  totally loose closure} of a graph $G=(V,E)$ in $\R^d$ as 
the minimal totally loose supergraph of $G$ in $\R^d$ and denote it by $\tlc_d(G)$. 
The next lemma shows that this notion is well-defined.

\begin{lemma}
Let $G=(V,E)$ be a graph. Then there is a unique minimal totally loose supergraph of $G$ in $\R^d$.
\end{lemma}
\begin{proof}
Suppose that we 
add all the edges $uv$ to $G$ for which $\{u,v\}$ is a non-adjacent weakly globally linked
pair in $G$ in $\R^d$, and then repeat this in the augmented graph, and so on, until
the resulting graph has no such pairs. Let these new edges be denoted by $e_1,\dots, e_k$ in the order in which they were added and let $E'=\{e_1,\dots, e_k\}$. Then $e_{i+1}$ is weakly globally linked in $G+\{e_1,\dots,e_i\}$ for $1\leq i\leq k-1$.  Hence, by induction on $i$, $e_i$ is an edge in every totally loose supergraph of $G$. 
Since $G+E'$ is totally loose, it follows that it is the unique minimal totally loose supergraph of $G$ in $\R^d$ (and hence we have $\tlc_d(G)=G+E'$).  
\end{proof}

The following lemma, which is a direct consequence of \cite[Lemma 3.2]{wgl}, is pivotal in this paper.

\begin{lemma}
\label{J}\cite{wgl}
Let $G=(V,E)$ be a graph and let $F$ be a set of edges in the complement of $E$. Then the following hold.\begin{enumerate}[(a)]\vspace{-0.2 cm}
\item $G+F$ is globally rigid in $\R^d$ if and only if $\tlc_d (G)+F$ is globally rigid in $\R^d$.\vspace{-0.2 cm}
\item $G$ is globally rigid in $\R^d$ if and only if %
$\tlc_d (G)$ is complete.
\end{enumerate}
\end{lemma}

As a corollary of Lemma \ref{J}(a), when we search for a minimum (cost or cardinality) set $F$ of new edges whose addition makes $G$ globally rigid
then we may replace $G$ with $\tlc_d (G)$. Hence, provided that we can efficiently compute the totally loose closure of a graph (which is the case in $\R^2$), it suffices to consider totally loose input graphs in the cheapest globally rigid supergraph problem. 
This is why in this paper totally loose graphs play a key role  and
  will be thoroughly analysed. 
 It will turn out that every rigid graph in $\R^2$ has a tree-like structure, which can be described with the help of its totally loose closure. 
By using an appropriate tree representation of this tree-like structure, we shall be able to reduce our augmentation problems to specific augmentation
problems on trees, and obtain the desired constant factor approximation algorithm.
We remark that the corresponding tree representation of a rigid graph in $\R^1$ is the so-called block-cut vertex tree.

\subsection{Basic properties of weakly globally linked pairs}
\label{sec:wgl2}

We %
 recall some basic lemmas from \cite{wgl}.

\begin{lemma}
\label{notlinked}\cite{wgl}
Let $G=(V,E)$ be a graph and let $\{u,v\}$ be a non-adjacent vertex pair.
If $\{u,v\}$ is not linked in $G$ in $\R^d$ then $\{u,v\}$ is globally loose in $G$ in $\R^d$.
\end{lemma}

Let $H=(V,E)$ be a graph and $x,y\in V$. We use $\kappa_H(x,y)$ to
denote the maximum number of pairwise internally disjoint $xy$-paths in
$H$. Note that if $xy\notin E$ then, by Menger's theorem,
$\kappa_H(x,y)$ is equal to the size of a smallest set $S\subseteq
V-\{x,y\}$ for which there is no $xy$-path in $H-S$.

\begin{lemma}
\label{kappa}\cite{wgl}
Let $G=(V,E)$ be a graph and let $\{u,v\}$ be a non-adjacent vertex pair
with $\kappa_G(u,v)\leq d$. Then $\{u,v\}$ is globally loose in $G$ in $\R^d$.
\end{lemma}

Let $G_i=(V_i,E_i)$ be a graph, $t \geq 1$ an integer, and suppose that $K_i$ is a complete subgraph 
(also called a {\it clique}) of $G_i$
on $t$ vertices, for $i=1,2$. Then the {\it $t$-clique sum} operation on $G_1,G_2$,
along $K_1,K_2$, creates a new graph $G$ by identifying the vertices of $K_1$ with the vertices
of $K_2$, following some bijection between their vertex sets. 
The {\it clique sum} operation is a $t$-clique sum operation for some $t\geq 1$.

\begin{lemma}
\label{cliquesum}\cite{wgl}
Suppose that $G$ is the clique sum of $G_1$ and $G_2$ and let $u,v\in V(G_1)$.
Then $\{u,v\}$ is weakly globally linked in $G$ in $\R^d$ if and only if
$\{u,v\}$ is weakly globally linked in $G_1$ in $\R^d$.
\end{lemma}

An edge $e$ of a globally rigid graph $H$ is {\it critical} if $H-e$ is not
globally rigid.

\begin{lemma} \cite{wgl}
\label{critical}
Let $G=(V,E)$ be a graph and $\{u,v\}$ be a non-adjacent vertex pair in $G$. Suppose that
$G$ has a globally rigid supergraph in $\R^d$ in which $uv$ is a critical edge.
Then $\{u,v\}$ is globally loose in $G$ in $\R^d$.
\end{lemma}

A basic graph operation is the {\it contraction} of a vertex set $X\subseteq V$ in a graph $G = (V, E)$. This
operation removes the vertices of $X$ from $G$, adds a new vertex $x$, and for each neighbour $w$ of $X$ in $G$
adds a new edge $xw$. The resulting graph is denoted by
$G/X$. The contraction
of an edge $e = xy$ means the contraction of the set $\{x, y\}$. It may also be denoted by $G/e$.

We will frequently use the following lemma. See Figure \ref{fig:gen9} for an illustration.
\begin{lemma}
\label{coro}\cite{wgl}
Let $G=(V,E)$ be a graph, $u,v\in V$, $V_0\subseteq V$. Suppose that $G[V_0]$ is a $(u,v)$-rigid subgraph of $G$ in $\R^d$, and suppose that
there is a $uv$-path in $G$ that is internally disjoint from $V_0$. Then
$\{u,v\}$ is weakly globally linked in $G$ in $\R^d$.
\end{lemma}
\begin{figure}[ht]
\centering
\includegraphics[scale=0.8]{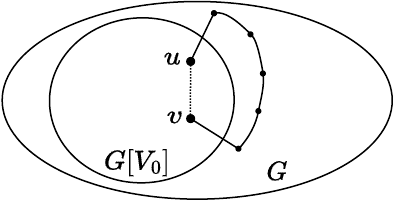}
\caption{Lemma \ref{coro} gives a sufficient condition for the weak global linkedness of $\{u,v\}$.}
\label{fig:gen9}
\end{figure}

\subsection{Weakly globally linked pairs in $\R^2$}

In this subsection we first state the key results from \cite{wgl} that are used in this paper.
Each of these results is concerned with the case $d=2$ and most of them will be used in the rest of this subsection only to
 prove Lemma \ref{wglvsgl} and Theorem
\ref{thm:onestep}.

A pair $(a, b)$ of vertices of a 2-connected graph $H = (V, E)$ is called a {\it 2-separator} if
$H - \{a, b\}$ is disconnected.

\begin{lemma}\cite{wgl}
\label{pair}
Let $G=(V,E)$ be a 2-connected graph and let $(u,v)$ be a 2-separator of $G$.
If the pair $\{u,v\}$ is linked in $G$ in $\R^2$, then $\{u,v\}$ is weakly globally linked in $G$ in $\R^2$.
\end{lemma}

\begin{theorem}
\label{globlaza jellemzes2}\cite{wgl}
Let 
$G=(V,E)$ be a $3$-connected graph and $u,v\in V$ with $uv\notin E$.
Suppose that $G_0=(V_0,E_0)$ is a subgraph of $G$ with $u,v\in V_0$ such that $G_0+uv$ is an ${\cal R}_2$-circuit. Then 
$\{u,v\}$ is weakly globally linked in $G$ in $\R^2$ if and only if 
${\rm Clique}(G,V_0)$ is globally rigid in $\R^2$.
\end{theorem}

\begin{lemma}\label{linkedpairs and 2separ}\cite{wgl}
Let $G=(V,E)$ be a 2-connected graph and let $\{u,v\}$ be a linked pair of vertices in $G$. Suppose that $(a,b)$ is a 2-separator of $G$. Let $C$ be a connected component of $G-\{a,b\}$, and let $V_0=V(C)\cup \{a,b\}$. Suppose that $u,v\in V_0$. Then $\{u,v\}$ is weakly globally linked in $G$ if and only if $\{u,v\}$ is weakly globally linked in $G[V_0]+ab$.
\end{lemma}

Let $G=(V,E)$ be a 2-connected graph, let 
$(a,b)$ be a 2-separator in $G$, and let 
$C$ be a connected component of $G-\{a,b\}$.
We say that the graph $G[V(C)\cup \{a,b\}]+ab$ (when $ab\notin E$) or  $G[V(C)\cup \{a,b\}]$ (when $ab\in E$)
is obtained from $G$ by a
{\it cleaving operation} along $(a,b)$. The graph $\bar G$ obtained from $G$ by adding every edge $ab$, for which $ab\notin E$ and $(a,b)$ is a 2-separator of $G$,
is called the {\it augmented graph} of $G$.

\begin{lemma}
\label{lem:cl}\cite{wgl}
Let $G=(V,E)$ be a 2-connected graph and let $\{u,v\}$ be a non-adjacent vertex pair in $G$
with $\kappa_G(u,v)\geq 3$. Then either $(u,v)$ is a 2-separator in $G$ or there is a unique
maximal 3-connected subgraph $B$ of $\bar G$ with $\{u,v\}\subset V(B)$.
In the latter case
the subgraph $B$ can be obtained from $G$ by a sequence of cleaving operations.
Furthermore, $uv\notin E(B)$, and
if the pair $\{u,v\}$ is linked in $G$ then it is also linked in $B$.
\end{lemma}

The subgraph $B$ in Lemma \ref{lem:cl} is called the {\it 3-block} of $\{u,v\}$ in $G$.

Let $G=(V,E)$ be a graph, $\emptyset\not= X\subseteq V$, and let $V_1,V_2,\dots, V_r$ be the
vertex sets of the connected components of $G-X$.
The graph ${\rm Clique}(G,X)$ is obtained from $G$ by deleting the vertex sets $V_i$, $1\leq i\leq r$, and adding
a new edge $xy$ for all pairs $x,y\in N_G(V_i)$, $xy\notin E$, for $1\leq i\leq r$.

The following theorem provides a complete characterization of the non-adjacent weakly globally linked pairs $G$.
By Lemma \ref{notlinked} and Lemma \ref{kappa}
we may assume that $\{u,v\}$ is linked and $\kappa_G(u,v)\geq 3$ (for otherwise $\{u,v\}$ is globally loose).
By Lemma \ref{cliquesum} we may also assume that
$G$ is 2-connected.

\begin{theorem}\cite{wgl}\label{cleavunit} %
Let $G=(V,E)$ be a 2-connected graph and let $\{u,v\}$ be a non-adjacent linked pair of vertices
with $\kappa_G(u,v)\geq 3$.
 Then $\{u,v\}$ is weakly globally linked in $G$ if and only if either\\
(i) $(u,v)$ is a 2-separator in $G$, or\\ 
(ii) ${\rm Clique}(B,V_0)$ is globally rigid,\\
where $B$ is the 3-block of $\{u,v\}$ in $G$, and $B_0=(V_0,E_0)$ is a subgraph
of $B$ with $u,v\in V_0$ such that $B_0+uv$ is an ${\cal R}_2$-circuit.
\end{theorem}

\noindent We next verify some further
basic properties of globally linked and globally loose pairs in $\R^2$.

\begin{lemma}\label{appendixlemma}
If $\{u,v\}$ is a globally loose linked pair in a  3-connected graph $B$, then $B$ has a globally rigid supergraph in which $uv$ is a critical edge
\end{lemma}

\begin{proof}
Since $\{u,v\}$ is linked in $B$, there exists a subgraph $B_0=(V_0,E_0)$ of $B$ with $u,v\in V(B_0)$
for which $B_0+uv$ is an ${\cal R}_2$-circuit.
By Theorem \ref{globlaza jellemzes2}, ${\rm Clique}(B,V_0)$ is not globally rigid.
Let $\bar B$ denote the supergraph of $B$ obtained by adding new edges that make
each component $C_i$ of 
$B-B_0$ a complete graph and make each vertex of $C_i$ connected to each vertex of 
$N_B(V(C_i))$.
Since $\bar B$ is 3-connected and ${\rm Clique}(\bar B,V_0)={\rm Clique}(B,V_0)$,
$\{u,v\}$ is globally loose in $\bar B$ by Theorem  \ref{globlaza jellemzes2}. 
Therefore, $\bar{B}$ is not globally rigid.
$B_0+uv$ is ${\cal R}_2$-connected and, since the addition of new edges or a vertex of degree at least three preserves ${\cal R}_2$-connectivity, $\bar{B}+uv$ is also ${\cal R}_2$-connected.
By using that $\bar B$ is $3$-connected, Theorem \ref{thm} implies that $\bar B+uv$ is a globally rigid supergraph of $B$ in which $uv$ is a critical edge.
\end{proof}

\begin{lemma}
\label{wglvsgl}
Let $G=(V,E)$ be a graph and $x,y,u,v\in V$. Suppose that $\{u,v\}$ is globally loose and $\{x,y\}$ is weakly globally linked in $G$ in $\R^2$. Then $\{u,v\}$ is globally loose in $G+xy$ in $\R^2$.
\end{lemma}

\begin{proof}%
Suppose, for a contradiction, that there exists a graph $G=(V,E)$ in which
\begin{itemize}
\item[($\ast$)] $\{u,v\}$ is globally loose, $\{x,y\}$ is weakly globally linked, and adding $xy$ to the graph makes $\{u,v\}$ weakly globally linked.
\end{itemize}
It is easy to see that we may assume that $G$ is 2-connected.  Since $\{u, v\}$ is weakly globally linked in $G + xy$, $\{u, v\}$ is also linked in $G + xy$. We claim that $\{u, v\}$ is linked in $G$, too. This follows from the fact that $\{x, y\}$ is weakly globally linked, and hence linked, in $G$, so the addition of $xy$ to $G$ cannot make a non-linked pair linked. 
Let $B$ denote the 3-block of $\{u,v\}$. We claim that $(\ast)$ holds in $B$ as well. Consider a 2-separator $(a,b)$. Since $\{x,y\}$ is weakly globally linked in $G$, $(a,b)$ does not separate $x$ and $y$ in $G$. Similarly, $(a,b)$ does not separate $u$ and $v$ in $G+xy$. It follows that $(a,b)$ does not separate $u$ and $v$ in $G$ either. Therefore $(\ast)$ remains true after cleaving the graph along $(a,b)$ by Lemma 
\ref{linkedpairs and 2separ}. %
Thus by replacing $G$ with $B$ we may assume that $G$ is 3-connected. %
Then, by Lemma \ref{appendixlemma}, $G$ has a supergraph $\bar G$ for which $\bar G+uv$ is globally rigid and $\bar G$ is not. $\bar{G}+uv+xy$ is also globally rigid and so is $\bar{G}+xy$  by $(*)$ and Lemma \ref{J}. Hence, $\{x,y\}$ is globally loose in $\bar G$ by Lemma \ref{critical} and thus also in $G$, a contradiction.
\end{proof}

  We obtain the following corollary of Lemma \ref{wglvsgl}.

\begin{theorem}
\label{thm:onestep}
Let $G=(V,E)$ be a graph. %
 Then $$\tlc_2(G)=G+J,$$ where $J=\{uv : uv\notin E, \{u,v\} \ \hbox{is weakly globally linked in}\ G \hbox{ in } \R^2\}.\vspace{-0.2cm}$
\end{theorem}

\section{Totally loose graphs}
\label{sec:Tloose}

Let us fix the dimension $d\geq 1$. Recall that a graph $G=(V,E)$ is called  totally loose if 
$\{u,v\}$ is globally loose in $G$ for all $u,v\in V$ with $uv\notin E$. In this section we prove a number of structural results on totally loose graphs.
Let us also introduce three related graph families.
A rigid graph is called $d$-{\it special}, if every rigid proper induced subgraph of $G$ is complete\footnote{$2$-special
graphs were simply called special graphs in \cite{JJS}, where they were first defined as a subfamily of the minimally rigid graphs in $\R^2$.
Note that in a minimally rigid graph every rigid subgraph is induced.
}. A graph $G$ is said to be {\it saturated non-globally rigid}, or SNGR, if $G$
is not globally rigid but $G+uv$ is globally rigid for all $u,v$ with $uv\notin E$.
SNGR graphs are rigid (by Theorem \ref{thm}) and totally loose (by Lemma \ref{critical}).

\subsection{A construction of rigid totally loose graphs in $\R^d$}

In this subsection we show that every rigid totally loose graph in $\R^d$ can be built up from $d$-special graphs
by clique sum operations. First we prove two lemmas.

\begin{lemma}
\label{glpf}
Let $G=(V,E)$ be a totally loose graph in $\R^d$. Then the following hold. 
\begin{enumerate}[(a)]
\item For every rigid subgraph $G[V_0]$ and each component $C$ of $G-V_0$ the set
$N_G(V(C))$ induces a clique in $G$.
\item If $G[V_1]$ and $G[V_2]$ are rigid subgraphs of $G$ with $V_1\cap V_2\neq \emptyset$, then $G[V_1\cap V_2]$ is rigid.
\end{enumerate}

\end{lemma}
\begin{proof}
 (a) directly follows from Lemma \ref{coro}. To prove (b) suppose that $G_1=G[V_1]$ and $G_2=G[V_2]$ are rigid subgraphs of $G$, and let $V_0=V_1\cap V_2$, $G_0=G[V_0]$. Let $U_1,\dots, U_k$ be the vertex sets of the components of $G_1-G_0$. By the rigidity of $G_2$ and statement (a), $N_{G_1}(U_i)\cap V_0$ induces a clique in $G_0$ for all $i\in \{1,\dots, k\}$. Hence, the rigidity of $G_1$ implies that $G_0$ is also rigid.
\end{proof}

In a graph $G=(V,E)$ a proper subset $U$ of $V$ is called a {\it clique separator} if $G[U]$ is a clique and $G-U$ is not connected. 
The graph obtained from the complete graph on $d+2$ vertices by deleting an edge is denoted by  $K_{d+2}-e$.

\begin{lemma}\label{rspeciallemma}
Let $G=(V,E)$ be a rigid totally loose graph in $\R^d$.
\begin{enumerate}[(a)]
\item Suppose that $|V|\leq d+2$. Then $G$ is either complete or isomorphic to $K_{d+2}-e$. Thus $G$ is $d$-special.
\item Suppose that $|V|\geq d+3$. Then $G$ is $d$-special if and only if there is no clique separator in $G$.

\end{enumerate}
\end{lemma}
\begin{proof} 
(a) follows from the facts that a graph $G$ on at most $d+1$ vertices is rigid if and only if $G$ is complete, and that $K_{d+2}-e$
is minimally rigid in $\R^d$.

(b) To prove the “if" direction suppose that there is no clique separator in $G$. Let $U$ be a proper subset of $V$ for which $G[U]$ is rigid. We shall prove that $U$ induces a clique. Let $C$ be a component of $G-U$. If $N_G(V(C))\subsetneq U$, then $N_G(V(C))$ is a clique separator by Lemma \ref{glpf}(a), which is a contradiction. Thus $N_G(V(C))=U$, and hence $U$ induces a clique by Lemma \ref{glpf}(a). %

To prove the “only if" direction suppose that $G$ is $d$-special and a proper subset $U$ of $V$ is a clique separator. 
Let $C_1,C_2$ be two components of $G-U$ and let $c_i\in V(C_i)$, $i=1,2$.
Since $G$ is rigid and $R$-special, the subgraphs $G[V(C_i)\cup U]$, $i=1,2$, are complete and
$|U|\geq d$ holds. Let $U'$ be a subset of $U$ with $|U'|=d$. Then $G[U'\cup \{c_1,c_2\}]$ is
rigid but not complete. Since $G$ is $d$-special, we must have $U=U'$ and $V=U\cup \{c_1,c_2\}$.
Thus $G$ is isomorphic to $K_{d+2}-e$.
\end{proof}

\begin{theorem}\label{Tlooseconstr}
Let $G$ be a rigid totally loose graph in $\R^d$. Then $G$ can be obtained from $d$-special graphs 
by a sequence of clique sum operations.
\end{theorem}
\begin{proof}
The proof is by induction on $|V|$. If $|V|\leq d+2$, then $G$ is $d$-special by Lemma \ref{rspeciallemma}(a), and hence the statement is obvious. Suppose that $|V|\geq d+3$. If $G$ is $d$-special, then we are done. Suppose that $G$ is not $d$-special. Then there is a clique separator $V_0\subsetneq V$ in $G$ by Lemma \ref{rspeciallemma}(b). Let $V_1$ be the vertex set of some component of $G-V_0$. Let $G'=G-V_1$ and $G''=G[N_G(V_1)\cup V_1]$. As $G'$ and $G''$ are induced subgraphs of a totally loose graph, both of them are totally loose.
Since $V_0$ is a clique separator and $G$ is rigid,
both $G'$ and $G''$ are rigid.
It follows by induction that $G'$ and $G''$ can be obtained from $d$-special graphs by a sequence of clique sum operations.
 $G$ can be obtained
from $G'$ and $G''$ by a clique sum operation, and thus the theorem follows.
\end{proof}

In the rest of this section we consider the case $d=2$.

\subsection{Totally loose graphs in $\R^2$}

In this subsection we strengthen the results on totally loose graphs and study the relations between $2$-special, SNGR, and totally loose graphs in $\R^2$.
See also
\cite{GJunitball,JJS} for some related
results.

\begin{lemma}
\label{gs}
Let $G=(V,E)$ be a $2$-special graph. Then one of the following holds:\\
(i) $G$ is a complete graph,\\
(ii) $G=K_4-e$,\\
(iii) $G$ is a $3$-connected SNGR graph in $\R^2$.
\end{lemma}

\begin{proof}
If $G$ has a $2$-separator $(a,b)$, then it is easy to see that $ab\in E$ and $G$ is the $2$-clique sum of two complete graphs of size three
along $(a,b)$. Thus we have $G=K_4-e$.
So we may assume that $G$ is $3$-connected and $|V|\geq 5$.

Suppose that $G$ is globally rigid. Then either $G$ is a complete
graph or there exists a non-adjacent pair
$u,v\in V$. In the latter case let us pick
an edge $e$ (resp. $f$) in $G$ incident with
$u$ (resp. $v$).
Since $G$ is globally rigid, it is ${\cal R}_2$-connected by Theorem \ref{thm}. 
Hence there is an
${\cal R}_2$-circuit $C$ in $G$ with $e,f\subset E(C)$. 
Since $|E(C)|=2|V(C)|-2$ and $d_C(v)\geq 3$ for all $v\in V(C)$, it follows that 
$C$ has at least four vertices of degree three. Thus we can find a vertex $x$ of degree three in $C$ different from $u,v$.
Now $C-x$ is rigid, so $V(C)-x$ induces a rigid proper subgraph of $G$ that contains
$u$ and $v$, which contradicts the
$2$-special
property of $G$.
Thus $G$ is globally rigid
only if $G$ is complete.

Next suppose that $G$ is not globally rigid.
Consider a pair $u,v\in V$ with $uv\notin E$.
Since $G$ is $2$-special, $H+uv$ is a
spanning ${\cal R}_2$-circuit of $G$ for any minimally rigid spanning subgraph $H$ of $G$. Thus $G+uv$ is ${\cal R}_2$-connected. 
Hence $G+uv$ is globally rigid by Theorem \ref{thm}.
So in this case it follows that $G$ is a $3$-connected SNGR graph.
\end{proof}

\begin{samepage}
\noindent Thus for rigid graphs in $\R^2$ we have the following chain of containment relations.
The examples in Figure \ref{fig:gen7} show that each of these containment relations is strict.

\medskip
$\ \ \ \ \ \ \ \ \ \ \ \ \ \ $  $2$-special \ $\subseteq$ \ SNGR or complete \  
$\subseteq$ \ totally loose 
\end{samepage}

\begin{figure}[ht]
\centering
\includegraphics[scale=1.5]{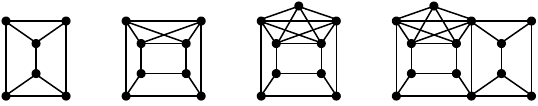}
\caption{The two  graphs on the left are rigid and $2$-special. The third graph is SNGR in $\R^2$, while the
rightmost graph  is totally loose in $\R^2$. }
\label{fig:gen7}
\end{figure}

\newpage
Next, we give a characterization of totally loose graphs in $\R^2$. Consider a $t$-clique sum operation of the graphs $G_1=(V_1,E_1)$ and $G_2=(V_2,E_2)$ along the cliques $K_1,K_2$ with $t\geq 2$ and $K_i\subseteq G_i$ for $i=1,2$. We say that this operation is a {\it restricted clique sum operation} if $t= 2$ or if  $t\geq 3$ and 
there exists an $i\in \{1,2\}$ for which $K_i$ is a maximal clique of $G_i$.

\begin{theorem} 
\label{glpf:con}
Let $G=(V,E)$ be a connected graph. Then $G$ is a rigid totally loose graph in $\R^2$ if and only if
$G$ can be obtained from $3$-connected SNGR graphs in $\R^2$ and complete graphs by a sequence of restricted clique sum
operations.
\end{theorem}

\begin{proof}
To prove the “only if" direction suppose that $G$ is rigid and totally loose. By Theorem \ref{Tlooseconstr} and Lemma \ref{gs}, $G$ can be obtained from $3$-connected SNGR graphs and complete graphs by a sequence of clique sum
operations. (Observe that
taking a clique sum with
a $K_4-e$ is equivalent to taking clique sums with
one or two $K_3$'s.) All these operations must be restricted clique sum
operations for otherwise $G$ would have an induced subgraph isomorphic to the globally rigid graph
$K_{t+2}-e$ with $t\geq 3$, which would be a contradiction since $G$ is totally loose. 

To prove the “if" direction suppose that $G$ can be obtained from $3$-connected SNGR graphs and complete graphs by a sequence of restricted clique sum
operations.
We prove that $G$ is rigid and totally loose by
induction on the number $h$ of clique sum operations. For $h=0$, $G$ is a complete graph or an SNGR graph, and hence 
$G$ is rigid and totally loose.  Suppose that $h\geq 1$.  
Consider the last clique sum operation, which is performed on two graphs $G_1=(V_1,E_1), G_2=(V_2,E_2)$, along a clique $S \subset V$, and which results in graph $G$. By induction $G_1$ and $G_2$ are rigid totally loose graphs. Thus the rigidity of $G$ follows. 

Suppose, for a contradiction, that $\{u,v\}$ is weakly globally linked in $G$ for some $u,v\in V$  with $uv\notin E$. Since $G_1$ and $G_2$ are totally loose, we may suppose that $u\in V_1-S$ and $v\in V_2-S$ by Lemma \ref{cliquesum}. 
Let $U_1$ (resp. $U_2$) be the vertex set of the component of $G-S$ containing $u$ (resp. $v$). If $|N_G(U_1)|\leq 2$ or $|N_G(U_2)|\leq 2$, then $\{u,v\}$ is globally loose in $G$ by Lemma \ref{kappa}, which is a contradiction. So we may suppose that $|N_G(U_i)|\geq 3$ for $i\in\{1,2\}$.
Let $q\in |N_G(U_2)|$.
Since $G_1$ is rigid, $\{q,u\}$ is weakly globally linked in $G+uv$ by Lemma \ref{coro}. Applying Lemma \ref{wglvsgl} with $\{u,v\}:=\{q,u\}$, $xy :=uv$ gives that $\{q,u\}$ is weakly globally linked in $G$ as well. 
By Lemma \ref{cliquesum} $\{q,u\}$ is also weakly globally linked in $G_1$. Since $G_1$ is totally loose, $qu\in E$. 
Hence, $|N_G(U_2)|\geq 3$ implies that for any $s\in S$, $\{s,u\}$ belongs to a globally rigid subgraph of $G_1$, and hence $\{s,u\}$ is globally linked in $G_1$. Thus $su\in E$. Therefore $S$ is not a maximal clique in $G_1$.
A similar argument shows that $S$ is not a maximal clique in $G_2$. Thus the last clique sum operation is not restricted,
a contradiction.
\end{proof}

\subsection{Standard subgraphs of totally loose graphs in $\R^2$}

Let $G=(V,E)$ be a rigid totally loose graph in $\R^2$. An induced subgraph $G[V_0]$  of $G$ (and the set $V_0$) is said to be {\it standard} in $G$ if it is rigid and for all $w\in V-V_0$, $|N_{G}(w)\cap V_0|\leq 2$. 
For example, for every edge $uv\in E$, $\{u,v\}$ is standard in $G$, and a clique $S$ is standard in $G$ if and only if $S$ is a maximal clique or $|V(S)|\leq 2$.   
Furthermore, if $G$ is 3-connected, then 
a clique is a standard clique separator if and only if it is a separating maximal clique in $G$.
It follows from Lemma \ref{glpf}(b) that for standard subgraphs $G_1,G_2$ of $G$ with $V(G_1)\cap V(G_2)\neq \emptyset$, $G_1\cap G_2$ is standard. 
 Suppose that $G_0$ is a standard subgraph of $G$. Then an induced  subgraph $H$ of $G_0$ is standard in $G_0$ if and only if it is standard in $G$.
See Figure  \ref{fig:gen3}.

\begin{figure}[ht]
\centering
\includegraphics[scale=0.9]{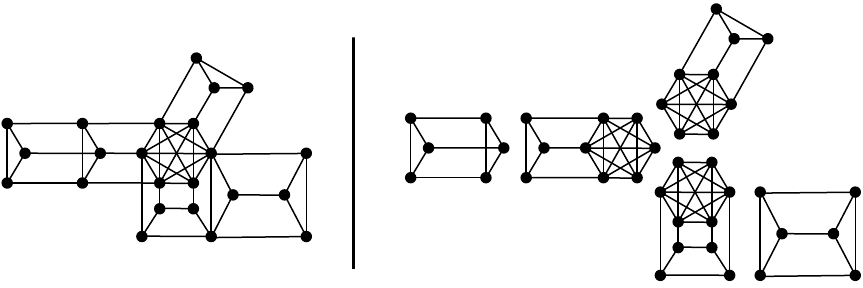}
\caption{The graph $G$ on the left %
is a rigid totally loose graph. On the right %
the  minimal non-complete standard subgraphs of $G$ are depicted. (Overall $G$ has 24 non-complete standard subgraphs. Five of these are depicted here. The remaining 19 can be obtained by taking the union of some of these five subgraphs such that the resulting graph is connected.)}
\label{fig:gen3}
\end{figure}

As we will see in the next subsection, the minimal non-complete standard subgraphs of a rigid totally loose graph can be considered its building blocks. In this subsection we prove that these subgraphs are always SNGR and some other auxiliary statements.

\begin{lemma}
\label{sum of complete graph and sngr graph is sngr}
Let $G=(V,E)$ be a  totally loose graph. 
 Suppose that $G'$ is a 3-connected induced subgraph of $G$ that can be obtained from a complete graph $S$ and an SNGR graph $H$ by a clique sum operation. 
Then $G'$ is an SNGR graph. 
\end{lemma}
\begin{proof}
We can consider $S$ and $H$ to be the appropriate induced subgraphs of $G$.
Let $S_0=N_{G'}(V(H)-V(S))$. Since $G'$ is 3-connected, $|S_0|\geq 3$. Let $u,v\in V(G')$ with $uv\notin E$. We have to show that $G'+uv$ is globally rigid. 
If $\{u,v\}\subset V(H)$, then $H+uv$ is globally rigid, and hence $G'+uv$ is also globally rigid. 
So we may assume that $u\in V(H)-V(S)$ and $v\in V(S)-V(H)$. Suppose, for a contradiction, that $us\in E$ for all $s\in S_0$. Then $\{u,v\}$ is globally linked in $G'$, and thus also in $G$, 
which contradicts the fact that $G$ is totally loose. Thus there exists some $s_0\in S_0$ with $us_0\notin E$.
Since $H$ is SNGR, $G'+us_0$ is globally rigid. 
By Lemma \ref{coro} and by the rigidity of $H$, $\{u,s_0\}$ is weakly globally linked in $G'+uv$. It follows by Lemma \ref{J} that $G'+uv$ is globally rigid.
\end{proof}

\begin{lemma}\label{3conn SNGR lemma}
Let $G=(V,E)$ be a non-complete rigid totally loose graph in $\R^2$. Then the following are equivalent. \begin{enumerate}[(i)]\vspace{-0.2 cm}
\begin{samepage}
\item $G$ is a 3-connected SNGR graph. \vspace{-0.2 cm}
\item $G$ has no standard clique separator.\vspace{-0.2 cm}
\item $G$ has no non-complete proper standard subgraph.
\end{samepage}
\end{enumerate} 
\end{lemma}

\begin{proof}
(i) $\Rightarrow$ (ii) Suppose that $G$ is a 3-connected SNGR graph, and suppose, for a contradiction, that $S$ is a standard clique separator in $G$. Let $C_1$ and $C_2$ be two different components of $G-S$. By 3-connectivity $|N_G(V(C_i))|\geq 3$ for $i=1,2$. Since $S$ is standard, $G_i=G[V(S)\cup V(C_i)]$ is non-complete; let $\{u_i, v_i\}$ be a non-adjacent vertex pair in $G_i$ for $i=1,2$. By Lemma \ref{cliquesum} $\{u_1,v_1\}$ is globally loose in $G+u_2v_2$, and thus $G+u_2v_2$ is not globally rigid by Lemma \ref{J}, a contradiction.

(ii) $\Rightarrow$ (i) Suppose that there is no standard clique separator in $G$. Since every 2-separator of a totally loose graph is a standard clique separator, $G$ is 3-connected. Let $V_0=\{v\in V:N_G(v) \text{ is not a clique in } G\}$. $G[V_0]$ can be obtained from $G$ by deleting the vertices whose neighbour set induces a clique. Hence, $G[V_0]$ is also a non-complete 3-connected rigid totally loose graph. 
Suppose, for a contradiction, that there exists a clique separator $S$ in $G[V_0]$. Then $S$ separates $G$ as well, and thus $S$ is not a standard clique in $G$. Let $S'$ be a standard clique in $G$ that contains $S$. 
Let $C_1$ denote the component of $G-S$ that contains $S'-S$ and let $C_2,\dots, C_k$ denote the remaining components of $G-S$. 
Since $S'$ is standard, it is not a clique separator in $G$. It follows that $k=2$ and $C_1=S'-S$.
Hence, $S$ is not a clique separator of $G[V_0]$, a contradiction. Since $G[V_0]$ has no clique separator, $G[V_0]$ is $2$-special by Lemma \ref{rspeciallemma}(b), and thus SNGR by Lemma \ref{gs}. As $G$ can be obtained from $G[V_0]$ by adding cliques to it, $G$ is SNGR by Lemma \ref{sum of complete graph and sngr graph is sngr}.

(ii) $\Rightarrow$ (iii)  Suppose that $G_0$ is a non-complete proper standard subgraph of $G$. Let $C$ be one of the components of $G-G_0$. Let $S=N_G(V(C))$. By Lemma \ref{glpf}(a) $S$ induces a clique in $G$. By 3-connectivity we have $|S|\geq 3$. Let $S'$ be a maximal clique of $G$ that contains $S$.  Since $G_0$ is standard, $S'\subseteq G_0$. $G_0$ is non-complete; hence, $S'$ is a standard clique separator of $G$. 

(iii) $\Rightarrow$ (ii) Suppose that $S$ is a standard clique separator of $G$. Let $C$ be a component of $G-V(S)$. Then $G[V(C)\cup V(S)]$ is a proper standard subgraph of $G$.
\end{proof}

\begin{lemma}\label{max 3-conn sngr}
Let $G=(V,E)$ be a rigid totally loose graph, and let $G_0$ be a subgraph of $G$. Then the following are equivalent.
\begin{enumerate}[(i)]\vspace{-0.2 cm}
\item $G_0$ is a maximal induced 3-connected SNGR subgraph of $G$. \vspace{-0.2 cm}
\item $G_0$ is a standard 3-connected SNGR subgraph of $G$. \vspace{-0.2 cm}
\item $G_0$ is a minimal non-complete standard subgraph of $G$.
\end{enumerate}
\end{lemma}
\begin{proof}

(i) $\Rightarrow$ (ii) follows from Lemma \ref{sum of complete graph and sngr graph is sngr}.

(ii) $\Rightarrow $ (iii) follows from Lemma \ref{3conn SNGR lemma}.

(iii) $\Rightarrow$ (i) 
Let $G_0$ be a minimal non-complete standard subgraph of $G$. It follows from Lemma \ref{3conn SNGR lemma} that $G_0$ is a 3-connected SNGR graph.  Suppose that $G_0'$ is a maximal 3-connected SNGR subgraph of $G$ with $G_0\subseteq G_0'$. Then $G_0$ is standard in $G_0'$, but $G_0'$ has no proper standard subgraph  by Lemma \ref{3conn SNGR lemma}. Hence, $G_0=G_0'$.%
\end{proof}

We continue the analysis with two further lemmas that we shall use in the next subsection.
 Let $G=(V,E)$ be a rigid totally loose graph.
Let \begin{align*}
{\cal H}_G&=\{H:H \text{ is a standard 3-connected SNGR subgraph of } G\}.
\end{align*}

\begin{lemma}\label{tloose subgraphs ind}
Let $G=(V,E)$ be a 3-connected rigid totally loose graph. Suppose that $S$ is a standard clique separator of $G$, and $G$ is the clique sum of $G'$ and $G''$ along $S$. Then ${\cal H}_{G}={\cal H}_{G'}\cup {\cal H}_{G''}.$
\end{lemma}
\begin{proof}
 To prove that ${\cal H}_{G}= {\cal H}_{G'}\cup {\cal H}_{G''}$ holds, consider the condition of Lemma \ref{max 3-conn sngr}(iii). If $G_0$ is a minimal non-complete standard subgraph of $G'$ (or $G''$), then it is also minimal in $G$. Thus we have ${\cal H}_{G}\supseteq {\cal H}_{G'}\cup {\cal H}_{G''}$. Let $G_0$ be a minimal non-complete standard subgraph (or, equivalently, a standard 3-connected SNGR subgraph) of $G$.  Suppose, for a contradiction, that $G_0\not\subseteq G', G_0\not \subseteq G''$. $S_0=G_0\cap S$ is the intersection of two standard subgraphs of $G$; hence, $S_0$ is standard in $G$, and also in $G_0$.
Thus $S_0$ is a standard clique separator of $G_0$, which is a contradiction by Lemma \ref{3conn SNGR lemma}. Thus either $G_0\subseteq G'$ and $G_0\in  {\cal H}_{G'}$, or $G_0\subseteq G''$ and $G_0\in {\cal H}_{G''}$. Hence, ${\cal H}_{G}\subseteq {\cal H}_{G'}\cup {\cal H}_{G''}$.
\end{proof}

\begin{lemma}
\label{separ in tloose subgraph}
Let $G=(V,E)$ be a  totally loose graph and let $H$ be a rigid induced subgraph of $G$ with $|V(H)|\geq 4$. 
\begin{enumerate}[(a)]
\item 
Suppose that $X$ is a separating set of $H$. Then $X$ is a separating set of $G$.
\item Suppose that $G$ is 3-connected.
Then $H$ is 3-connected.

\item Suppose that $G$ is 3-connected, and $S_1$ and $S_2$ are %
cliques in $G$ with $|V(S_1)\cap V(S_2)|\geq 2$. Then $V(S_1)\cup V(S_2)$ also induces a clique in $G$.

\end{enumerate}
\end{lemma}
\begin{proof}
(a) Suppose, for a contradiction, that $X$ does not separate $G$. Then there is path $P$ in $G-X$ that connects the vertex sets of two different components of $H-X$. By taking a subpath, we may suppose that $P$ is internally disjoint from $H$. By Lemma \ref{coro}, the first and last vertices of $P$ are weakly globally linked in $G$. Since $G$ is totally loose, these two vertices are also connected in $G$, which is a contradiction.

(b) follows from (a).

 (c) If $|V(S_1)\cap V(S_2)|\geq 3$, then $G[V(S_1)\cup V(S_2)]$ is globally rigid, and thus complete.
 Suppose that $|V(S_1)\cap V(S_2)|= 2$. By (b) there are vertices $u\in V(S_1)-V(S_2), v\in V(S_2)-V(S_1)$ with $uv\in E$. Thus $G[V(S_1)\cup V(S_2)]$ is globally rigid again, and thus complete.
\end{proof}

\subsection{The tree structure of rigid totally loose graphs  in $\R^2$}\label{sec:treerepdef}

In this subsection we conclude our analysis regarding the structure of rigid totally loose graphs. 
We  shall prove that they possess a tree-like structure, in which the building blocks are the standards
3-connected SNGR subgraphs and the standard complete subgraphs.
Let $G$ be a 2-connected graph.
 A maximal 2-connected subgraph of $G$ with no 2-separators is called a {\it 3-block} of $G$. %
  Note that in a rigid totally loose graph for every 2-separator $(a,b)$, we have $ab\in E$  by Lemma \ref{pair}.

\begin{lemma}
\label{lem:triangleor3conn}
Let $G=(V,E)$ be a rigid totally loose graph with $|V|\geq 3$. Then $G$ is the clique sum of its 3-blocks. Each 3-block of $G$ is a triangle or a rigid 3-connected graph. Each 3-block is standard in $G$.
\end{lemma}
\begin{proof}
The proof is by induction on $|V|$. If $G$ is 3-connected, then the statement is obvious. If not, then there is a 2-separator $(a,b)$ in $G$. Let $V_1, \dots, V_k$ denote the vertex sets of the components of $G-\{a,b\}$. Let $G_i=G[V_i\cup \{a,b\}]$ for $1\leq i\leq k$. The set of the 3-blocks of $G$ is the union of the sets of the 3-blocks of $G_i$, $1\leq i \leq k$. Since $ab\in E$ by Lemma \ref{pair}, $G$ is the clique sum of $G_1, \dots, G_k$ along $(a,b)$, and the lemma follows by induction by noting that
each $G_i$ is rigid and totally loose.
\end{proof}

\begin{lemma}\label{unique min standard}
Let $G=(V,E)$ be a rigid totally loose graph with $|V|\geq 3$. Suppose that $Q=(a,b)$ is a 2-separator of $G$ and $G_0$ is a 3-block of $G$ with $Q\subseteq G_0$. %
Let $\mathcal{X}$ denote the set of the standard subgraphs of $G_0$ that contain $Q$ and have at least three vertices. Then $\mathcal{X}$ has a unique minimal element $X$. 
$X$ is either a standard clique or a standard 3-connected induced SNGR subgraph of $G$.
\end{lemma}
\begin{proof}
Suppose, for a contradiction, that $X_1$ and $X_2$ are two different minimal elements of $\mathcal{X}$. $X_1\cap X_2$ is a standard subgraph of $G_0$ that contains $Q$, and thus it follows from the minimality of $X_1$ that $X_1\cap X_2=Q$. By Lemma \ref{separ in tloose subgraph}(b) $G[V(X_1)\cup V(X_2)]$ is 3-connected. Thus there is some $x\in V(X_1)-\{a,b\}$, $y\in V(X_2)-\{a,b\}$ such that $xy\in E$. Since $X_2$ is rigid, $\{a,y\}$ and $\{b,y\}$ are weakly globally linked in $G$ by Lemma \ref{coro}. As $G$ is totally loose, $ay,by\in E$. Thus the existence of $y$ contradicts the fact that $X_1$ is standard. Therefore, $X$ is indeed unique. Since $X$ is standard in $G_0$ and $G_0$ is standard in $G$, $X$ is also standard in $G$.
By the minimality of $X$, it has no standard clique separator. Hence, if $X$ is not a clique, then it is a 3-connected SNGR graph by Lemma \ref{3conn SNGR lemma}.
\end{proof}

\begingroup
\allowdisplaybreaks
Let $G=(V,E)$ be a rigid totally loose graph. Let \begin{align*}
 {\cal S}_G&=\{S:S \text{ is a standard clique of } G \text{ with } |V(S)|\geq 3\},\\
 {\cal Q}_G&=\{Q:Q \text{ is a 2-separator of } G\}.
\end{align*}

 We define the {\it tree representation} $T_G$ of $G$ as the graph $T_G=({\cal H}_G\cup {\cal S}_G\cup {\cal Q}_G, F)$, where there are two kinds of edges in $F$:
 \begin{itemize}
 \item For $H\in {\cal H}_G, S\in {\cal S}_G$ we have $SH\in F$ if $S\subseteq H$.
 \item  For every $Q\in {\cal Q}_G$ and 3-block $G_0$ of $G$ with $Q\subseteq G_0$, we connect $Q$ with the vertex that represents the minimal standard subgraph $X$ of $G_0$ that contains $Q$ and has at least three vertices. This subgraph is unique by Lemma \ref{unique min standard}. See Figure \ref{fig:gen4}. 
 \end{itemize}

\begin{figure}[t]
\centering
\includegraphics[scale=0.9]{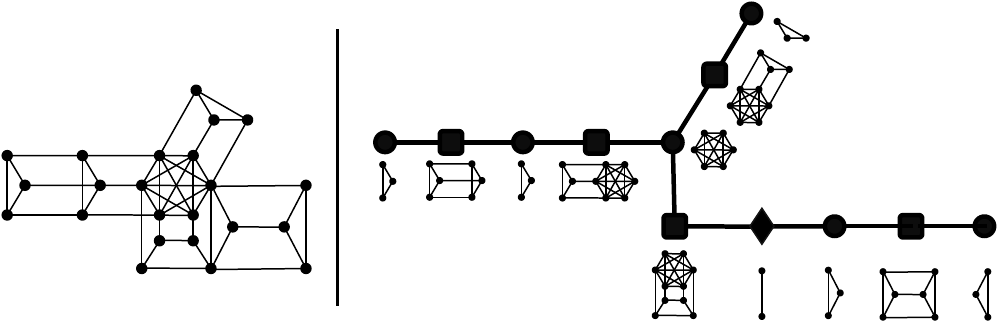}
\caption{A rigid totally loose graph and its tree representation. Each vertex of the tree representation represents a subgraph of $G$, which is depicted next to the vertex. The three different types of vertices  in the tree representation are indicated by the different shapes. (square - 3-connected SNGR graph; circle - clique; diamond - 2-separator)}
\label{fig:gen4}
\end{figure}

For a subgraph $T_0$ of $T_G$, let $G[T_0]$ denote the subgraph of $G$ obtained by taking the union of the subgraphs of $G$ corresponding to the vertices of $T_0$.%
\begin{lemma}\label{tree rep is a tree}
Let $G=(V,E)$ be a rigid totally loose graph with $|V|\geq 3$. Then the tree representation $T_G=(\mathcal{H}_G\cup \mathcal{S}_G\cup \mathcal{Q}_G, F)$ of $G$ is a tree. For every subtree $T_0$ of $T_G$, $G[T_0]$ is a standard subgraph of $G$ that can be obtained from the set of graphs corresponding to $V(T_0)$ by clique sum operations between the (subgraphs that are represented by) adjacent vertices of $T_0$. Furthermore, $G[T_G]=G$.
\end{lemma} 
\begin{proof}

The proof is by induction on $|V|$. If $G$ is complete, then $\mathcal{H}_G= \mathcal{Q}_G=\emptyset$, $\mathcal{S}_G=\{G\}$, and thus the statement is trivial. Consider the case where $G$ is a non-complete 3-connected rigid totally loose graph. If there is no standard clique separator in $G$, then $G$ is SNGR by Lemma \ref{3conn SNGR lemma}, $\mathcal{H}_G=\{G\}$, $\mathcal{Q}_G=\emptyset$.  In this case $T_G$ is a star. Suppose that there is a standard clique separator $S$ in $G$. 
Let $G'$ and $G''$ be induced subgraphs of $G$ such that $G$ is the clique sum of $G'$ and $G''$ along $S$. By induction $T_{G'}$ and $T_{G''}$ are trees.
It is clear that ${\cal S}_{G}={\cal S}_{G'}\cup {\cal S}_{G''}$, and by Lemma \ref{tloose subgraphs ind} we have ${\cal H}_{G}={\cal H}_{G'}\cup {\cal H}_{G''}.$ Hence,  $T_G$ can be obtained by gluing together $T_{G'}$ and $T_{G''}$  along the vertex $S$. For every subtree $T_0$ of $T_G$ with $S\in V(T_0)$, $G[T_0]$ can be obtained by taking the clique sum of $G[T_0\cap T_{G'}]$ and $G[T_0\cap T_{G''}]$ along the clique $S$. Thus the auxiliary statements follow by induction. 

Now suppose that $G$ is not 3-connected and let $Q=(a,b)$ be a 2-separator of $G$. Suppose that $(c,d)$ is another 2-separator. By Lemma \ref{pair}, $cd\in E$, and it follows %
that $(a,b)$ does not separate $c$ and $d$. Let $V_1,\dots, V_k$ be the vertex sets of the components of $G-\{a,b\}$. Let $G_i=G[V_i\cup \{a,b\}]$ for $1\leq i\leq k$. 
By induction $T_{G_i}$ is a tree for $1\leq i\leq k$. Furthermore, $$\mathcal{Q}_G =\{Q\}\cup \bigcup_{i=1}^k\mathcal{Q}_{G_i},\;\;\;\; \mathcal{H}_G=\bigcup_{i=1}^k\mathcal{H}_{G_i} \;\; \text{ and } \;\; \mathcal{S}_G=\bigcup_{i=1}^k \mathcal{S}_{G_i}.$$
Hence,
$T_G$ can be obtained by taking the disjoint union of the trees $T_{G_i}$ ($1\leq i\leq k$), adding the vertex $Q$, and connecting $Q$ to each tree. These operations result in a tree. For every subtree $T_0$ of $T_G$ with $Q\in V(T_0)$, $G[T_0]$ can be obtained by taking the clique sum of the graphs $G[T_0\cap T_{G_i}]$ ($1\leq i\leq k$) along $Q$. Thus the auxiliary statements follow by induction. 
\end{proof}

According to Lemma \ref{tree rep is a tree}, for every subtree $T_0$ of $T_G$, $G[T_0]$ is a standard subgraph of $G$. The following lemma implies that every standard subgraph with at least 3 vertices can be obtained this way. 

\begin{lemma}
Let $G=(V,E)$ be a rigid totally loose graph, and let $G_0$ be a standard subgraph of $G$ with at least 3 vertices. Let $T_G$ (resp. $T_{G_0}$) be the tree representation of $G$ (resp $G_0$).   Then $T_{G_0}$ is a subtree of $T_G$, and %
$G_0=G[T_{G_0}]$. 
\end{lemma} 
\begin{proof}
Since every standard subgraph of $G_0$ is also standard in $G$, and every 2-separator of $G_0$ is also a 2-separator of $G$ by Lemma \ref{separ in tloose subgraph}(a), 
we have
$\mathcal{H}_{G_0}\subseteq \mathcal{H}_G, \mathcal{S}_{G_0}\subseteq \mathcal{S}_G \text{ and } \mathcal{Q}_{G_0}\subseteq \mathcal{Q}_G .$ 
Thus $V(T_{G_0})\subseteq V(T_G)$. The relation $E(T_{G_0})\subseteq E(T_G)$ follows from the definitions and the fact that if a standard subgraph $X$ of $G$ is contained in $G_0$, then it is standard in $G_0$ as well.
  \end{proof}

In certain augmentation problems 
it is useful to work with a simplified 
tree representation, which is obtained by removing some redundant leaves.
We shall use the following concept in Section \ref{sec:minsize}.
For a rigid totally loose graph $G$ let $T_G^r$ denote the minimal subtree of $T_G$ for which $G[T_G^r]=G$. $T_G^r$ can be obtained from $T_G$ by removing every leaf vertex of $T_G$ that is a subgraph of its neighbour in $T_G$. Note that all of these vertices are in ${\cal S}_G$, and for every leaf vertex $S$ of $T_G$ for which $S\in \mathcal{S}_G\cap V(T_G^r)$, $S$ is a 3-block of $G$. We call $T_G^r$ the {\it reduced tree representation} of $G$. 
We will simply denote the set $V(T_G^r)\cap {\cal S}_G$ (resp. ${\cal H}_G$, ${\cal Q}_G$) by $\mathcal{S}$ (resp. $\mathcal{H}, \mathcal{Q}$).

\section{Globally rigid supergraphs}
\label{sec:augmenations}

In this section we give a simple characterization of
the sets of new edges whose addition makes
a rigid totally loose graph globally rigid in $\R^2$.

\subsection{The totally loose closure of an augmented graph}\label{subsec:tlcaug}

 Let $G=(V,E)$ be a rigid totally loose graph in $\R^2$. 
For $Z\subseteq V$ let $\cll(G,Z)$ (resp. $\cll^*(G,Z)$) denote a minimal $Z$-rigid (resp. minimal standard $Z$-rigid) vertex set in $G$. By Lemma \ref{glpf}(b), $\cll(G,Z)$ is unique. %
By applying the following lemma to $V_0=\cll(G,Z)$, it follows that $\cll^*(G,Z)$ is unique, too. 

\begin{lemma}
Let $G=(V,E)$ be a rigid totally loose graph in $\R^2$ and let $V_0 \subseteq V$ be a vertex set for which $G[V_0]$ is rigid. Then $\cll^*(G,V_0)$ can be obtained from $V_0$ by adding the vertices $w$ to $V_0$ for which $|N_G(w)\cap V_0|\geq 3$.
\end{lemma}
\begin{proof}
Let $V_1$ denote the vertex set that is obtained from $V_0$ by adding the vertices $w$ to $V_0$ for which $|N_G(w)\cap V_0|\geq 3$. It is obvious that $V_1\subseteq \cll^*(G,V_0)$. To see that $V_1\supseteq \cll^*(G,V_0)$ we shall prove that $V_1$ is a standard set.
Suppose, for a contradiction, that there is some $u\in V-V_1$ such that $|N_G(u)\cap V_1|\geq 3$. Let $X$ denote $N_G(u)\cap V_1$. Since $u\notin V_1$, $X\not\subseteq V_0$. Thus there is some $x\in X-V_0$. It follows from the rigidity of $G[V_1\cup\{u\}-\{x\}]$ by Lemma \ref{coro} that $uy\in E$ for all $y\in N_G(x)\cap V_0$. Thus $N_G(x)\cap V_0\subseteq N_G(u)\cap V_0$.  Since $x\in V_1$, we have $|N_G(x)\cap V_0|\geq 3$. It follows that $|N_G(u)\cap V_0|\geq 3$, which is a contradiction.
\end{proof}

For a vertex set $Z$, the complete graph on $Z$ is denoted by $K(Z)$. 

\begin{theorem}\label{tlc GplusK(Z)}
Let $G=(V,E)$ be a rigid totally loose graph in $\R^2$, $Z\subseteq V$. Then $$\tlc(G+K(Z))=G+K(\cll^*(G,Z)).$$
\end{theorem}
\begin{proof}
If $|Z|= 1$, then the statement is trivial. Suppose that $|Z|\geq 2$. We claim that $\bar G=G+K(\cll^*(G,Z))$ is totally loose. $G[V(C)\cup N_G(V(C))]$ is a rigid totally loose graph for every component $C$ of $\bar{G}-\cll^*(G,Z)$. $\bar{G}$ can be obtained by adding these totally loose graphs to the clique $K(\cll^*(G,Z))$. Since $\cll^*(G,Z)$ is standard, these clique sums are restricted. Thus $\bar G$ is indeed totally loose by Theorem \ref{glpf:con}. It follows that $\bar G\supseteq  \tlc(G+K(Z))$.

We shall prove that $\bar{G} \subseteq \tlc(G+K(Z))$ holds as well.  Let $V_0=\cll(G,Z)$, $H=G[V_0]$ and $\bar H=\tlc(G+K(Z))[V_0]$. Since $G\left[\cup_{u,v\in Z} (\cll(G,\{u,v\}))\right]$ is rigid, we have $$V_0=\bigcup_{u,v\in Z} (\cll(G,\{u,v\})).$$ We claim that $\bar H$ is a complete graph. If $|V_0|\leq 3$, then this follows from the rigidity of $H$.
 Suppose that $|V_0|\geq 4$.
  We first show that $\bar H$ is 3-connected. For any 2-separator $(a,b)$ of $H$, $ab\in E$ by Lemma \ref{pair}. Thus for every component $C$ of $H-\{a,b\}$, the graph $G[V(C)\cup \{a,b\}]$ is rigid. 
  Hence, the minimality of $\cll(G,Z)$ implies that there are $u,v\in Z$ such that $(a,b)$ separates $u$ and $v$ in $G$. 
It follows that $(a,b)$ is not a 2-separator in $ \bar H$, and thus $\bar H$ is indeed 3-connected. Suppose that $|Z|=2$, $Z=\{u,v\}$. 
Then, since $H$ has no proper $(u,v)$-rigid subgraph, $H_0+uv $ is an ${\cal R}_2$-circuit for every minimally rigid spanning subgraph $H_0$ of $H$; therefore, $H+uv$ is ${\cal R}_2$-connected and $\bar H$ is complete by Theorem \ref{thm}. 
Now suppose that $|Z|\geq 3$. Let $S$ be a maximal clique in $\bar H$ that contains $Z$. It follows from the $|Z|=2$ case that for every pair $u,v\in Z$, $\cll(G,\{u,v\})$ is a clique in $\bar{H}$. Hence, $S$ contains $\cll(G,\{u,v\})$ by Lemma \ref{separ in tloose subgraph}(c) and the 3-connectivity of $\bar{H}$.  Therefore, $V(S)=V_0$ and $\bar H$ is indeed complete. Hence, $G+K(V_0)\subseteq \tlc(G+K(Z))$. Since every totally loose supergraph of $G+K(V_0)$ is a supergraph of $\bar{G}$, we have $\bar{G} \subseteq \tlc(G+K(Z))$.
\end{proof}

\subsection{The characterization of globally rigid supergraphs}
\label{subsec:treerepandaug}

 Let $G=(V,E)$ be a rigid totally loose graph, let $u,v\in V$ with $uv\notin E$, and consider  the standard subgraph $G_0=G[\cll^*(G,\{u,v\})]$. 
Let $P$ be a minimal subtree of $T_G$ for which $G[P]=G_0$. We claim that $P$ is unique and that $P$ is a path.
Let $W^u$ be the set of vertices of $T_G$ that contain $u$. $W^u$ induces a subtree of $T_G$; let this subtree be denoted by $T^u$. Define $T^v$ similarly.
 Note that for a subtree $T'$ of $T_G$, $G_0\subseteq G[T']$ if and only if $T'\cap T^u\neq \emptyset\neq T'\cap T^v$. %
It follows that $P$ is the unique shortest path that connects $T^u$ and $T^v$. Let $X$ and $Y$ be the two end-vertices of $P$. ($T^u$ and $T^v$ might have a common vertex, but, since $uv\notin E$, not more. In this case $XY$ is a loop.) We say that $XY$ is the edge corresponding to $uv$, and we denote $XY$ by $T_G\langle uv\rangle$. For a set $E'$ of edges in the complement of $G$, let $T_G\langle E'\rangle=\{T_G\langle uv\rangle\}_{uv\in E'}$.
 
 \begin{theorem}\label{glob rigid augmentation lemma}
Let $G=(V,E)$ be a rigid totally loose graph, and let $T_{G}=(\mathcal{H}_G\cup \mathcal{S}_G\cup \mathcal{Q}_G, F)$ be the tree representation of $G$. Suppose that $E'$ is a set of edges in the complement of $G$. The graph $G+E'$ is globally rigid if and only if
\begin{enumerate}[(i)]
\item there is a cycle in $T_G+T_G\langle E'\rangle$ containing $H$ for every $H\in \mathcal{H}_G$, and
\item $Q$ is not a cut vertex in $T_G+T_G\langle E'\rangle$ for every $Q\in \mathcal{Q}_G$. 
\end{enumerate} 
\end{theorem}
\begin{proof}
To prove the “only if" direction suppose that $G+E'$ is globally rigid. Let $H\in \mathcal{H}_G$, and suppose, for a contradiction, that there is no cycle of $T_G+T_G\langle E'\rangle$ that contains $H$. Let $T_1,\dots, T_k$ denote components of $T_G-H$, and let $V_i=V(G[T_i])$. Consider the graph $\bar G=G+K(V_1)+\dots+K(V_k)$. $\bar G$ can be obtained from $\{H, K(V_1),\dots, K(V_k)\}$ by clique sum operations. Since every neighbour of $H$ in $T_G$ represents a standard clique in $G$, these clique sums are restricted. It follows that $\bar G$ is totally loose by Theorem \ref{glpf:con}. 
Since $G+E'\subseteq \bar G$, we have $\tlc(G+E')\subseteq \bar G$, which is a contradiction by Lemma \ref{J}, as $G+E'$ is globally rigid. Let $Q=\{a,b\}$ be a 2-separator of $G$. $E'$ induces a connected graph on the components of $G-\{a,b\}$. Hence, $T_G\langle E'\rangle$ induces a connected graph on the components of $T_G-Q$.  

To prove the “if" direction suppose that (i) and (ii) hold. It follows from (ii) that $G+E'$ is 3-connected. %
Note that since $T_G$ is a tree, 
(i) is equivalent to saying that for every $H\in \mathcal{H}$ there exists an edge $uv\in E'$ such that $P$ contains $H$, where $P$ is the path that connects the end-vertices of $T_G\langle uv\rangle$ in $T_G$. 
Then $V(H)\subseteq \cll^*(G,\{u,v\})$. 
By Theorem \ref{tlc GplusK(Z)}, $V(H)$ induces a clique in $\tlc(G+uv)$, and thus in $\tlc(G+E')$. 
Let $T_0$ be a maximal subtree of $T_G$ for which $V(G[T_0])$ induces a clique in $\tlc(G+E')$. Suppose, for a contradiction, that $T_0\neq T_G$. Let $X\in V(T_G)-V(T_0)$ be adjacent to $T_0$. Let $U_1=V(G[T_0])$ and $U_2=V(X)$. Both $U_1$ and $U_2$ induce a clique in $\tlc(G+E')$. Hence, $U_1\cup U_2$ also induces a clique in $\tlc(G+E')$ by Lemma \ref{separ in tloose subgraph}(c), which contradicts the maximality of $T_0$. Thus $T_0=T_G$ and $\tlc(G+E')$ is a complete graph. By Lemma \ref{J}, $G+E'$ is globally rigid.
\end{proof}

\section{Cheapest globally rigid supergraphs}
\label{sec:cheapest}

In this section, we give a 5-approximation algorithm for the optimization problem of finding the cheapest globally rigid supergraph of an arbitrary input graph. Formally, the input of the problem is a graph $G=(V,E)$ and a cost function $c:V\times V\to \R^+\cup \{\infty\}$. Our goal is to find a cheapest set $E'$ of new edges for which $G'=(V,E\cup E')$ is globally rigid in $\R^2$.
Consider the following problems.
\begin{itemize}[label={}, leftmargin =0.6cm] 
\item \underline{Problem A} \\Find a cheapest rigid supergraph of a given graph.

\vspace{-0.1cm}\item \underline{Problem B} \\Find a cheapest supergraph of a given tree $T$ in which every element of a given set $U$ of vertices is not a cut vertex.

\vspace{-0.1cm}\item \underline{Problem C} \\Find a cheapest supergraph of a given tree $T$ in which every element of a given set $U$ of vertices is contained in a cycle.
\end{itemize}

Problem A is the problem of finding a cheapest base of the rigidity matroid on the vertex set $V$.  Hence, it can be solved optimally in polynomial time. However, both Problem B and C are NP-complete: By putting $U=V(T)$, the cheapest 2-vertex-connected supergraph problem can be reduced to Problem B. And it is an easy exercise to show that the cheapest 2-edge-connected supergraph problem can be reduced to Problem C. 

\begin{lemma}\label{2approxB}
There is a 2-approximation algorithm for Problem $B$.  
\end{lemma}
\begin{proof}
Let $T=(V,E)$ be a tree, $U\subset V$. Let $\bar T$ be the supergraph of $T$ in which we connect all those vertex pairs $u,v\in V$ for which there exists a $u$-$v$-path in $T$ that is internally disjoint from $U$. For a set $F$ of edges, $\bar T+F$ is 2-connected if and only if every element of $U$ is not a cut vertex in $T+F'$. Thus we can reduce Problem B to an equivalent 2-vertex-connected supergraph problem for which we can use the 2-approximation algorithm of \cite{KhRa}.
\end{proof}

\begin{lemma}\label{2approxC}
There is a 2-approximation algorithm for Problem C.
\end{lemma}
\begin{proof}  Suppose that a tree $T=(V,F)$, a set $U\subseteq V$ and a cost function $c:V\times V\to \R^+\cup\{\infty\}$ are given. Let $F'$ denote the set of edges with finite cost. For $f\in F'$ let $P_f$ be the unique path in $T$ that connects the end-vertices of $f$. Let $\mathcal{P}=\{P_f:f\in F'\}$. Extend $c$ to $\mathcal{P}$ so that $c(P_f)=c(f)$. Problem $C$ is equivalent to finding a cheapest cover of $U$ by the paths in $\mathcal{P}$.

We first show that this problem is optimally solvable if $T$ is an arborescence, not a tree, and the paths are directed paths. For any $v\in V$ let $T_v$ be the set of $v$ and its descendants. 
Let ${\rm OPT}_v$ denote the cheapest cover of $T_v\cap U$.  
If $v\notin U$, then $${\rm OPT}_v=\sum_{r\in C(v)}{\rm OPT}_r,$$ where $C(v)$ denotes the children of $v$. If $v\in U$, then $${\rm OPT}_v = \min_{v\in P\in \mathcal{P}}\bigg\{c(P)+\sum_{r\in R(v,P)}{\rm OPT}_r\bigg\}, $$ where $R(v,P)$ denotes the set of those descendants of $v$ that are roots in the graph $T[V_v]-V(P)$. Thus a dynamic program, which starts from the leaves of the arborescence and computes ${\rm OPT}_v$ for every vertex $v$, solves the problem in polynomial time.  

Consider the original problem, where $T$ is a tree. Choose an arbitrary vertex $r\in V$, and orient the graph from $r$ as a root. Then every path $P_f\in \mathcal{P}$ falls apart into two directed paths $P'_f$ and $P''_f$. Define the cost of $P_f'$ and $P_f''$ as $c(P_f)$ and solve the covering problem with the paths $\bigcup_{P_f\in \mathcal{P}}\{P_f',P_f''\}$. Let the optimum solution of this modified problem be denoted by $\mathcal{S}$. The optimum value of the modified problem is at most two times the optimum value of the original problem. Thus we get a 2-approximation for the original problem by choosing every path $P_f$ for which at least one of $P_f'$ and $P_f''$ is in $\mathcal{S}$.
\end{proof}

We need one more notion before giving a constant-factor approximation algorithm for the cheapest globally rigid supergraph problem. 
Let $c:V\times V\to \R^+\cup \{\infty\}$ be a cost function and let $T$ be a tree representation of a graph with vertex set $V$. Let the {\it corresponding cost function} $c_T:V(T)\times V(T)\to \R^+\cup \{\infty\}$ be defined as $$c_T(XY)=\min_{T\langle uv\rangle=XY} c(uv)$$ for $X,Y\in V(T)$. Let $T^{-1}_c\langle XY\rangle$ denote the edge for which the minimum above is attained (if it is finite). For a set $F$ of edges on the vertex set $V(T)$, let $T^{-1}_c\langle F\rangle=\{T^{-1}_c\langle XY\rangle\}_{XY\in F}$.

\begin{algorithm}
\caption{}\label{alg:cap}
 \hspace*{\algorithmicindent} \textbf{Input:} a graph $G=(V,E)$ and a cost function $c:V\times V\to \R^+\cup \{\infty\}$\\
  \hspace*{\algorithmicindent} \textbf{Output:} an edge set $E'$ for which $G+E'$ is globally rigid in $\R^2$  
\begin{algorithmic}[1]
\State $\bar G=G+E^A$
$ \gets$ $c$-cheapest rigid supergraph of $G$ %
\Comment{instance of Problem A}\smallskip

\State $T=(\mathcal{H}\cup \mathcal{S}\cup \mathcal{Q}, F) \gets$ tree representation of the totally loose closure of $\bar G$  \smallskip

\State $c_T$ $\gets$ cost function on $V(T)\times V(T)$ corresponding to $c$ \smallskip

\State $T+F^B$
$\gets $ 2-approximation for the problem of finding a $c_T$-cheapest supergraph of $T$ in which for all $Q\in \mathcal{Q}$, $Q$ is not a cut vertex, 
\Comment{instance of Problem B}\smallskip

\State $T+F^C$
$\gets$ 2-approximation for the problem of finding a $c_T$-cheapest supergraph of $T$ in which for all $H\in \mathcal{H}$, $H$ is contained in a cycle
\Comment{instance of Problem C}\smallskip

\State \textbf{if} any of the optimum values above is infinite \textbf{then return $V\times V$}\smallskip

\State $E^B \gets T_c^{-1}\langle F^B\rangle$ %
\smallskip

\State $E^C \gets T_c^{-1}\langle F^C\rangle$ %
\smallskip

\State \textbf{return} $E^A\cup E^B\cup E^C$
\end{algorithmic}
\end{algorithm}

\begin{theorem}
Algorithm 1 is a 5-approximation for the cheapest globally rigid supergraph problem.
\end{theorem}
\begin{proof}
Let $\rm{OPT}_{GRA}$ denote the optimum of the cheapest globally rigid supergraph problem with the input $G,c$. Let $\rm{OPT}_A$ (resp. $\rm{OPT}_B$ and $\rm{OPT}_C$) denote the optimum of the problem in line 1 (resp. line 4 and line 5). Then $\rm{OPT}_A\leq \rm{OPT}_{GRA}$. By Theorem \ref{glob rigid augmentation lemma} the inequalities $\rm{OPT}_B\leq \rm{OPT}_{GRA}$ and $\rm{OPT}_C\leq \rm{OPT}_{GRA}$ hold as well.  In particular, if any of $\rm{OPT}_A$, $\rm{OPT}_B$ and $\rm{OPT}_C$ is infinite, then $\rm{OPT}_{GRA}$ is also infinite. Suppose that this is not the case. Then $\tlc(\bar G)+ (E^B\cup E^C)$ is globally rigid by Theorem \ref{glob rigid augmentation lemma} and thus $ G+(E^A\cup E^B\cup E^C)$ is globally rigid by Lemma \ref{J}.  Furthermore,
 $$c(E^A\cup E^B\cup E^C)\leq c(E^A)+c(E^B)+c(E^C)=$$ $$=c(E^A)+c(F^B)+c(F^C)\leq \rm{OPT}_A+2OPT_B+2\rm{OPT}_C\leq 5\rm{OPT}_{GRA}.$$ This completes the proof.
\end{proof}

\section{Minimum size globally rigid supergraphs of rigid graphs}
\label{sec:minsize}

In this section we consider the special case of the cheapest supergraph problem
in which each new edge that we add to the input graph $G$ has unit cost. In this case we look for 
a minimum size globally rigid supergraph (or augmentation) of $G=(V,E)$ in $\R^d$. 
Kir\'aly and Mih\'alyk\'o \cite{KMsidma} gave a
polynomial time algorithm and a min-max formula for the optimum in the case when $d=2$ and $G$ is rigid in $\R^2$.
Our goal is to obtain similar result by using a substantially different approach, based on the new structural results
developed in the previous sections.

We start by reviewing the case $d=1$, which corresponds to the well known 2-connectivity augmentation problem.
Recall that a graph $G$ is 2-connected if it has at least two vertices and
$G-v$ is connected for all $v\in V(G)$. Note that $K_2$ is 2-connected.
Let $G=(V,E)$ be a connected graph with at least two vertices. For $v\in V$ let $b(G,v)$ denote the number of components of the graph $G-v$ and let $b(G)=\max \{b(G,v):v\in V\}$. 
A vertex $v$ is called a {\it cut vertex} of $G$ if $b(G,v)\geq 2$. 
A maximal 2-connected subgraph of $G$ is a {\it block}. 
A block of $G$ is called an {\it end-block} if it contains at most one cut vertex of $G$. 
The {\it block-cut vertex tree} of $G$ has a vertex for each cut vertex of $G$ as well as for each block of $G$; a (vertex representing a) cut vertex $v$ and a (vertex representing a) block $H$ is connected by an edge in the block-cut vertex tree if and only if $v\in V(H)$. It is well-known that it is indeed a tree. Let $t(G)$ denote the number 
of end-blocks of $G$, i.e. the number of leaves of its block-cut vertex tree. 
Eswaran and
Tarjan \cite{ET}, and independently Plesnik \cite{Plesnik}, showed that the 2-connectivity augmentation problem is solvable in
polynomial time and provided the following characterization.

\begin{theorem}\label{2-conn aug of graphs}\cite{ET, Plesnik}
Let $G=(V,E)$ be a connected graph with $|V|\geq 3$.
Then $$\min (|E'|: G+E' \text{ is 2-connected}) = \max\Set*{\bigg\lceil\frac{t(G)}{2}\bigg\rceil, b(G)-1}.$$
\end{theorem}

We shall also need the corresponding result for the 2-edge connectivity augmentation problem.

\begin{theorem}\label{2-edgeconn aug of graphs}\cite{ET}
Let $G=(V,E)$ be a connected graph. Then $$\min (|E'|: G+E' \text{ is 2-edge-connected}) = \bigg\lceil \frac{t(G)}{2}\bigg\rceil.$$
\end{theorem}

\subsection{Globally rigid augmentation in $\R^2$}

For a subset $U$ of $V$, let $b_U(G)=\max \{b(G,v):v\in U\}$. Furthermore, for $\{a,b\}\subset V$, let $c(G,\{a,b\})$ denote the number of components of the graph $G-\{a,b\}$ and let $c(G)=\max\{c(G,\{a,b\}):a,b\in V\}$.

 Let $G=(V,E)$ be a rigid graph in $\R^2$ and assume that $G+E'$ is globally rigid in $\R^2$. There are two simple lower bounds for $|E'|$.  Since every globally rigid graph is 3-connected, $E'$ must induce a connected graph on the components of $G-\{a,b\}$ for every pair $a,b\in V$, and thus $|E'|\geq c(G,\{a,b\})-1$. Hence, $|E'|\geq c(G)-1$. 
The following concept yields another lower bound for $|E'|$. We call a set $U\subseteq V$ {\it untied} if the graph $G+K({V-U})$ is not globally rigid and we call  $U$ an {\it untied end} if it is a minimal untied set. Let $l(G)$ denote the maximum number of pairwise disjoint untied sets in $G$. Clearly, $V(E')$ must intersect every untied set in $G$. Thus $|E'|\geq \lceil l(G)/2 \rceil$.

This gives us the $\min \geq \max$ inequality of the following min-max theorem.

\begin{theorem}\label{glob rigid aug minmax thm}
Let $G=(V,E)$ be a rigid graph. Then 
$$\min ( |E'|:G+E'=(V,E\cup E') \text{ is globally rigid}) = \max\Set*{ \bigg\lceil \frac{l(G)}{2}\bigg\rceil, c(G)-1}.\;\;\;\;\;(*)$$ 
Let $\tlc (G)$ be the totally loose closure of $G$, and let $T_{\tlc (G)}^r=(\mathcal{H}\cup \mathcal{S}\cup \mathcal{Q}, F)$ be the reduced tree representation of $\tlc (G)$. 
If $\tlc (G)$ is not a complete graph or a 3-connected SNGR graph, then $(*)$ is further equal to $\max\Set*{ \big\lceil t\big(T_{\tlc (G)}^r\big)/2\big\rceil, b_\mathcal{Q}\big(T_{\tlc (G)}^r\big)-1}$.
\end{theorem}

Before proving the $\min \leq \max$ inequality of Theorem \ref{glob rigid aug minmax thm} we prove two lemmas.

\begin{lemma}\label{tree 2-edge-connected aug}
Let $T=(V, F)$ be a tree with $|V|\geq 2$, and let $\mathcal{Q}$ be a subset of $V$. Then 
$$\min\Set*{|F'|\given\longsetdescriptiontwo{
        $T+F'$ is 2-edge-connected, and for all\;\; $Q\in \mathcal{Q}$, $Q$ is not a cut vertex in $T+F'$
      }} =\max\Set*{\bigg\lceil\frac{t(T)}{2}\bigg\rceil, b_{\mathcal{Q}}(T)-1}.$$
\end{lemma}
\begin{proof}If $\mathcal{Q}$ does not contain any cut vertices, then we are done %
by Theorem \ref{2-edgeconn aug of graphs}.
 Suppose that $\mathcal{Q}$ contains at least one cut vertex.
 
Suppose that $T+F'$ is 2-edge-connected, and $Q$ is not a cut vertex in $T+F'$ for all $Q\in \mathcal{Q}$. Then $V(F')$ must contain the set of leaves, hence $|F'|\geq \lceil t(T)/2\rceil$. Besides, $F'$ must induce a connected subgraph on the components $G-Q$, and thus $|F|\geq b(T,Q)$, for every $Q\in \mathcal{Q}$. The $\min \geq \max$ inequality follows. 

To prove the other direction, for every vertex $v\in V-\mathcal{Q}$, replace $v$ by a cycle of length $\min(3, \deg_T(v))$, and connect the neighbours of $v$ to different vertices in the cycle. Let this new graph be denoted by $T'$. 
 It is easy to see that $b_\mathcal{Q}(T)=b(T')$ and $t(T)=t(T')$. Thus, by Theorem \ref{2-conn aug of graphs}, $T'$ can be made 2-connected by adding $\max\Set*{\lceil t(T)/2\rceil, b_{\mathcal{Q}}(T)-1}$ edges. By contracting the cycles we obtain a 2-edge-connected supergraph of $T$, in which $Q$ is not a cut vertex for all $Q\in \mathcal{Q}$.
\end{proof}

Let $G=(V,E)$ be a graph and $\tlc (G)$ be the totally loose closure of $G$. For a leaf $X$ of $T_{\tlc (G)}^r=(\mathcal{H}\cup \mathcal{S}\cup \mathcal{Q}, F)$, let $W_X$ be the set of those vertices of $X$ that do not belong to (the subgraph represented by) any other vertex of $T_{\tlc (G)}^r$. Note that for every leaf $X$ of $T_{\tlc(G)}^r$,  $X\in \mathcal{H}\cup \mathcal{S}$ and $W_X$ is non-empty. 
\begin{lemma}\label{untied sets are leafs}
Let $G=(V,E)$ be a graph and $\tlc (G)$ be the totally loose closure of $G$. Suppose that $\tlc (G)$ is not a complete graph or a 3-connected SNGR graph. Then $U\subseteq V$ is untied in $G$ if and only if there is a leaf $X$ of $T_{\tlc (G)}^r$ with $W_{X}\subseteq U$. Hence, the set of the untied ends of $G$ is $\{W_X:X$ is a leaf of $T_{\tlc (G)}^r\}$. 
\end{lemma}
\begin{proof}
By Lemma $\ref{J}$ a set is untied in $G$ if and only if it is untied in $\tlc (G)$, thus we can assume that $G=\tlc (G)$. Let $X$ be a leaf of $T_{G}^r$. If $X\in \mathcal{S}$,
 then $X$ is a 3-block of $G$ and $|N_G(W_X)|=2$; hence, $W_X$ is untied. If $X\in \mathcal{H}$, then $W_X$ is untied again, since $N_{G}(W_X)$ is a clique by Lemma \ref{glpf}(a). If a set is untied, then all its supersets are untied, which completes the proof of the “if" direction.

To prove the “only if" direction let $U\subseteq V$ be a set for which $W_X\not\subseteq U$ for every leaf $X$ of $T_{ G}^r$. For a leaf $X$ let $w_X\in W_X-U$. Let $E'$ be the edges of the complete graph on the vertex set $\{w_X: X$ is a leaf of $T_{G}^r\}$. %
Then $T^r_{G}+T_G\langle E'\rangle$ is 2-connected. Since $T_G^r$ is obtained from $T_G$ by deleting vertices that belong to $\mathcal{S}_G$, it follows that the conditions of Theorem \ref{glob rigid augmentation lemma} are fulfilled; $G+E'$ is globally rigid, an thus so is $G+K(V-U)$. Therefore $U$ is not untied.
\end{proof}

\begin{corollary}\label{lequals tTG}
Let $G$ be a graph and suppose that $\tlc(G)$ is not a complete graph or a 3-connected SNGR graph. Then the untied ends of $G$ are pairwise disjoint and $l(G)=t(T_{\tlc (G)}^r)$.
\end{corollary}

\begin{proof}[Proof of Theorem \ref{glob rigid aug minmax thm}]
If $\tlc (G)$ is a complete or a 3-connected SNGR graph, the theorem is easy to check. Suppose that this is not the case. We claim that the following inequalities hold.
\begingroup
\allowdisplaybreaks
\begin{align*}
       \min&\Set{|E'|\given G+E' \text{ is globally rigid}} \\[1.5ex]
       \mathrel{\mathop{=}\limits^{(1)}}\min&\Set{|E'|\given \tlc (G)+E' \text{ is globally rigid}} \\[1.5ex]
 	\mathrel{\mathop{\le}\limits^{(2)}}  \min&\Set*{|F'|\given\longsetdescription{
        $T_{\tlc (G)}^r+F'$ is 2-edge-connected, and for all $Q\in \mathcal{Q}$, $Q$ is not a cut vertex in $T_{\tlc (G)}^r+F'$
      }}  \\[1.5ex]
    \mathrel{\mathop{=}\limits^{(3)}} \max&\Set*{\bigg\lceil\frac{t\big(T_{\tlc (G)}^r\big)}{2}\bigg\rceil, b_{\mathcal{Q}}\big(T_{\tlc (G)}^r\big)-1} \\[1.5ex]
        \mathrel{\mathop{=}\limits^{(4)}} \max&\Set*{\bigg\lceil\frac{l(G)}{2}\bigg\rceil, c(G)-1}\\[1.5ex]
               \mathrel{\mathop{\le}\limits^{(5)}}\min&\Set{|E'|\given G+E' \text{ is globally rigid}} 
\end{align*}
\endgroup
(1) follows from Lemma \ref{J}, (3) follows from Lemma \ref{tree 2-edge-connected aug}, (4) follows from Corollary \ref{lequals tTG}, (5) was noted before Theorem \ref{glob rigid aug minmax thm}. (2) remains to be proven.

 Suppose that $T_{\tlc (G)}^r+F'$ is 2-edge-connected, and $Q$ is not a cut vertex in $T_{\tlc (G)}^r+F'$ for all $Q\in \mathcal{Q}$. We give an edge set $E'$ of size $|F'|$ for which $\tlc (G)+E'$ is globally rigid. We may assume that for all $f\in F'$, $f$ connects two different leaves of $T_{\tlc (G)}^r$. 
For $f=XY\in F$, let $u_f\in W_{X}$ and let $v_f\in W_{Y}$;
then $T_{\tlc(G)}\langle u_fv_f\rangle=XY$. Let $E'=\{u_fv_f\}_{f\in F'}$; then $T_{\tlc(G)}\langle E'\rangle=F'$.
It follows that the conditions of Theorem \ref{glob rigid augmentation lemma} are fulfilled; $E'$ is an edge set of size $|F'|$ for which $\tlc (G)+E'$ is globally rigid.

Hence, all the values of the inequality chain are equal.
\end{proof}

\section{Chordal graphs}
\label{sec:chordal}

Let $d\geq 1$ be fixed.
In this subsection we give a 2-approximation algorithm for the cheapest globally rigid supergraph problem for $d$-connected chordal input graphs in $\R^d$.  As we will see, this problem is equivalent to finding a minimum cost set of new edges that make the given $d$-connected chordal graph $(d+1)$-connected.
It is well-known that
 a graph $G$ is chordal if and only if its vertices have an ordering such that for each vertex $v$, the vertices that are adjacent to  $v$ and precede $v$ in the ordering induce a clique in $G$.
This fact easily implies the following characterization of $d$-connected chordal graphs.
 
 \begin{lemma}
 A graph $G=(V,E)$ is a $d$-connected chordal graph if and only if $G$ is the clique sum of complete graphs along cliques of size at least $d$.
 \end{lemma}  

We call this the clique sum decomposition of $G$. 
The following lemma describes the totally loose closures of $d$-connected chordal graphs.

\begin{lemma}\label{chordal char}
Let $G$ be a $d$-connected chordal graph. Then $\tlc_d(G)$ is the clique sum of complete graphs along cliques of size exactly $d$. 
\end{lemma}
\begin{proof}
The proof is by induction on $|V|$. If $|V|=d+1$, then $G$ is complete and thus the statement is obvious.
Suppose that $|V|\geq d+2$ and let $v\in V$ such that $N_G(v)$ induces a clique in $G$. Let $U=V-\{v\}$. Then $G[U]$ is a $d$-connected chordal graph. Hence, by induction, $\tlc_d(G[U])$ is the clique sum of complete graphs along cliques of size $d$.
If $|N_G(v)|=d$, then $\tlc_d(G)$ can be obtained from $\tlc_d(G[U])$ by adding a complete graph with vertex set $N_G(v)\cup \{v\}$ along the clique $N_G(v)$, and thus the statement follows. Suppose that $|N_G(v)|\geq d+1$. There is a complete graph $S$ in the clique sum decomposition of $\tlc_d(G[U])$ with $ N_G(v)\subseteq V(S)$. If we replace  $S$ with a complete graph on $V(S)\cup \{v\}$ in the clique sum decomposition of $\tlc_d(G[U])$, then we obtain the clique sum decomposition of $\tlc_d(G)$, which completes the proof.
\end{proof}

Lemma \ref{chordal char} implies that the totally loose closure of a $d$-connected chordal graph is also chordal.
Let $G$ be a totally loose $d$-connected chordal graph in $\R^d$.  Let $\mathcal{S}_G$ denote the set of the maximal cliques of $G$ and let $\mathcal{Q}_G$ denote the the $d$-separators of $G$. We define the {\it tree representation} $T_{G}$ as the graph $(\mathcal{S}_G\cup \mathcal{Q}_G,F)$ in which for $Q\in \mathcal{Q}_G, S\in \mathcal{S}_G$, we have $QS\in F$ if and only if $S$ contains $Q$.
For a subgraph $T_0$ of $T_G$ let $G[T_0]$ denote the subgraph of $G$ obtained by taking the union of the subgraphs of $G$ corresponding to the vertices of $T_0$. 
The following lemma is easy to prove by induction on the number of clique sums in the clique sum decomposition \mbox{of $G$}.

\begin{lemma}
Let $G=(V,E)$ be a totally loose $d$-connected chordal graph in $\R^d$. Then the tree representation $T_G=(\mathcal{S}_G\cup \mathcal{Q}_G,F)$ is a tree. For every subtree $T_0$ of $T_G$, $G[T_0]$ is an induced subgraph of $G$ that can be obtained from the set of graphs $V(T_0)$ by clique sum operations between the (subgraphs represented by) adjacent vertices of $T_0$. Furthermore, $G[T_G]=G$.
\end{lemma}

We will need the following $d$-dimensional version of Lemma \ref{separ in tloose subgraph}(c). 

\begin{lemma}\label{separd}
Let $G$ be a $(d+1)$-connected totally loose graph in $\R^d$. Suppose that $S_1$ and $S_2$ are %
cliques in $G$ with $|V(S_1)\cap V(S_2)|\geq d$. Then $V(S_1)\cup V(S_2)$ also induces a clique in $G$.
\end{lemma}
\begin{proof}
The proof is the same as for $d=2$.
\end{proof}

Let $G=(V,E)$ be a totally loose chordal graph, $u,v\in V$ with $uv\notin E$. Let $P$ denote the minimal path in $T_G$ for which $\{u,v\}\subseteq G[P]$ and let $X$ and $Y$ denote the end-vertices of $P$. 
We denote $XY$ by $T_G\langle uv\rangle$.
 For a set $E'$ of edges in the complement of $G$, let $T_G\langle E'\rangle=\{T_G\langle uv\rangle\}_{uv\in E'}$.
\begin{theorem}\label{chordalGR}
Let $G=(V,E)$ be a $d$-connected chordal graph with $|V|\geq d+2$ and let $E'$ be a set of edges in the complement of $G$. Then the following are equivalent.
\begin{enumerate}[(i),noitemsep,topsep=0pt]
\begin{samepage}
\item $G+E'$ is globally rigid in $\R^d$.
\item $G+E'$ is $d+1$-connected.
\item For every $Q\in \mathcal{Q}_{\tlc_d(G)}$, $Q$ is not a cut vertex in $T_{\tlc_d(G)}+T_{\tlc_d(G)}\langle E'\rangle$.
\end{samepage}
\end{enumerate}
\end{theorem}
\begin{proof}
(i) $\Rightarrow$ (ii) follows from the fact that every globally rigid graph on at least $d+2$ vertices is $d+1$-connected.

(ii) $\Leftrightarrow$ (iii) is obvious.

(iii) $\Rightarrow$ (i) By Lemma \ref{J}, we may assume that $G=\tlc_d(G)$. %
Let $T_0$ be a maximal subtree of $T_G$ for which $V(G[T_0])$ induces a clique in $\tlc(G+E')$. Suppose, for a contradiction, that $T_0\neq T_G$. Let $X\in V(T_G)-V(T_0)$ be adjacent to $T_0$. Let $U_1=V(G[T_0])$ and $U_2=V(X)$. Both $U_1$ and $U_2$ induce a clique in $\tlc(G+E')$. Hence, $U_1\cup U_2$ also induces a clique in $\tlc(G+E')$ by Lemma \ref{separd}, which contradicts the maximality of $T_0$. Thus $T_0=T_G$ and $\tlc(G+E')$ is a complete graph. By Lemma \ref{J}, $G+E'$ is globally rigid.
\end{proof}

By Lemma \ref{chordalGR}, the $d$-dimensional cheapest globally rigid supergraph problem for $d$-connected chordal input graphs can be reduced to Problem B from Section \ref{sec:cheapest}, for which we have a 2-approximation algorithm by Lemma \ref{2approxB}.

\section{Algorithmic aspects}
\label{sec:algo}

In this section we verify the polynomial running time of our 5-approximation algorithm Algorithm 1 for the cheapest globally rigid supergraph problem in $\R^2$.

Given the input graph $G=(V,E)$ and the cost function $c$ on the non-edges of $G$, our algorithm has five main steps:
(i) find a cheapest rigid supergraph $\bar G$ of $G$,
(ii) find
the totally loose closure $\tlc_2(\bar G)$, (iii) construct the tree representation $T$ of $\tlc_2(\bar G)$,
(iv) find a 2-approximation for an instance of Problem B with input $T$,
(v) find a 2-approximation for an instance of Problem C with input $T$. 
We shall argue that each of these steps can be implemented in polynomial time. Let $n=|V|$ and $m=|E|$.

Step (i) can be formulated as finding a minimum cost base in the 2-dimensional rigidity matroid of
the complete graph. 
Hence it can be solved by the greedy algorithm in $O(|V||E|)$ time, see e.g., \cite{Jmemoirs}.
It follows from Theorem \ref{thm:onestep} that $\tlc_2(\bar G)$ can be obtained from $\bar G$ by
adding every edge $uv$ for which $\{u,v\}$ is weakly globally linked in $\bar G$. It was shown in \cite{wgl} that
testing the weak global linkedness of a given pair can be done in $O(|V|^2)$ time. Thus step (ii) can
be performed in $O(|V|^4)$ steps.

Let us consider step (iii). For simplicity let $G'=\tlc_2(\bar G)$.
The tree representation $T$ of $G'$ can be obtained by repeatedly separating the graph along 2-separators and
(when the graph is 3-connected) along separating maximal cliques, and then merging the smaller trees as described
in the proof of Lemma \ref{unique min standard}. 
It is well-known that we can find all 2-separators in linear time. Recall that 
each 3-block of $G'$ is a triangle or a rigid 3-connected graph by
Lemma \ref{lem:triangleor3conn}. To see that the separating maximal cliques can be found quickly, and their number is limited,
we use Lemma \ref{separ in tloose subgraph}(c). It implies that for each edge $e$ of $G'$ there is a unique maximal clique $S_e$ that
contains $e$. Hence each member of the collection ${\cal S}=\{ S_e : e\in E(G')\}$ can be obtained in a greedy manner,
$|{\cal S}|\leq |E(G')|$, and the maximal clique found, in every iteration, must belong to ${\cal S}$.
This gives rise to an $O(|V||E|)$ algorithm for constructing $T$.

Before performing steps (iv) and (v) we compute the corresponding cost function $c_T$ (which is easy).
Then we first solve Problem B by using the reduction described in Lemma \ref{2approxB} and calling
the algorithm of \cite{KhRa} for the cheapest 2-vertex-connected supergraph problem.
Finally, we solve Problem C by the dynamic programming algorithm in linear time,
following the proof of Lemma \ref{2approxC}.

\section*{Acknowledgements}
This work was supported by the Hungarian Scientific Research Fund provided by the National Research, Development and Innovation Office, grant No. K135421. %
The first author was also supported in part by the MTA-ELTE Momentum Matroid Optimization Research Group and the
National Research, Development and Innovation Fund of Hungary, financed under the ELTE
TKP 2021‐NKTA‐62 funding scheme.

\end{document}